\documentclass[10pt]{amsart} 
\usepackage{tikz-cd}
\usepackage{fancyhdr} 
\usepackage{latexsym, amssymb, amscd, amsfonts,pstricks,enumitem,amsmath, young}
\usepackage{latexsym, amssymb, amscd, amsfonts,pstricks,enumitem,amsmath}
\usepackage{amsmath,amsthm,amssymb,mathrsfs}
 \usepackage[normalem]{ulem}
\usepackage[all]{xypic}
\usepackage{ifthen}
\usepackage{longtable} 
\usepackage{todonotes,marginnote}

\renewcommand{\subsection}[1]{\vspace{3mm}\refstepcounter{subsection}\noindent{\bf \thesubsection. #1.} }

\renewcommand{\subsubsection}[1]{\vspace{3mm}\refstepcounter{subsubsection}\noindent{\bf \thesubsubsection. #1.} }

\numberwithin{equation}{section}

\providecommand{\binom}[2]{{#1\choose#2}}

\renewcommand{\geq}{\geqslant}
\renewcommand{\leq}{\leqslant}
\newcommand{\Osh}{{\mathcal O}}                        

\newcommand{\G}{\mathrm{G}}

\newcommand{\kk}{\mathbf{k}}

\newcommand{\ord}{\mathrm{ord}}

\newcommand{\CC}{\mathbb{C}} 
\newcommand{\NN}{\mathbb{N}} 
\newcommand{\PP}{\mathbb{P}} 
\newcommand{\QQ}{\mathbb{Q}} 
\newcommand{\ZZ}{\mathbb{Z}} 

\newtheorem{theorem}{Theorem}[section]
\newtheorem{lemma}[theorem]{Lemma}

\newtheorem{proposition}[theorem]{Proposition}

\theoremstyle{definition}
\newtheorem{defn}[theorem]{Definition}
\newtheorem{remark}[theorem]{Remark}
\newtheorem{example}[theorem]{Example}

\begin{document}
\title[GCD with moving targets and linear recurrence sequences]{Greatest common divisors with moving targets and  consequences for linear recurrence sequences
}

\begin{abstract}
We establish consequences of the moving form of Schmidt's Subspace Theorem.  Indeed, we obtain inequalities that bound the logarithmic greatest common divisor of moving multivariable polynomials evaluated at moving $S$-unit arguments.  In doing so, we complement recent work of Levin.   As an additional application, we obtain results that pertain to the greatest common divisor problem for algebraic linear recurrence sequences.   These observations are motivated by previous related works of Corvaja-Zannier, Levin and others.
 \end{abstract}
 \thanks{2000\ {\it Mathematics Subject Classification}: Primary 11J87; Secondary 11B37, 11J25}
\thanks{The second author was supported in part by Taiwan's MoST grant 108-2115-M-001-001-MY2.}
\thanks{This article has been published in TAMS.  The final published version is available at: https://doi.org/10.1090/tran/8220}
\normalsize
\baselineskip=14pt

\maketitle

\section{Introduction}

In the recent work \cite{Levin:GCD}, Levin obtained the following result which bounds the greatest common divisor of multivariable polynomials.  This result (Theorem \ref{Levin:thm1.1} below) generalizes earlier results of Bugeaud-Corvaja-Zannier \cite{Bugeaud:Corvaja:Zannier:2003},   Hern\'andez-Luca  \cite{Hernandez:Luca:2003} and Corvaja-Zannier \cite{Corvaja:Zannier:2003}, \cite{Corvaja:Zannier:2005}.  We refer to \cite{Levin:GCD} for a survey of these related results.
\begin{theorem}[{\cite[Theorem 1.1]{Levin:GCD}}]\label{Levin:thm1.1}
Let $\Gamma \subseteq \G^r_m(\overline{\QQ})$ be a finitely generated group and fix nonconstant coprime polynomials $f(x_1,\dots,x_r), g(x_1,\dots, x_r) \in \overline{\QQ}[x_1,\dots,x_r]$ which do not both vanish at the origin $(0,\dots,0)$.  Then, for each $\epsilon > 0$, there exists a finite union $Z$ of translates of proper algebraic subgroups of $\G_m^r$ so that
$$
\log \gcd (f(\mathbf{u}), g(\mathbf{u}) ) < \epsilon \max_i \{h(u_i)\} 
$$
for all $\mathbf{u} = (u_1,\dots,u_r) \in \Gamma \setminus Z$.
\end{theorem}
The greatest common divisor on the left-hand side of the above inequality is a generalized notion of the usual quantity for integers, adapted to algebraic numbers \cite[Definition 1.4]{Levin:GCD}.
As an application of the above theorem, Levin classified  when terms from simple linear recurrence sequences can have a largest common divisor.

The main purpose of this article is to obtain  a \emph{moving form}, in the sense of \cite{Ru:Vojta:1997}, of Theorem \ref{Levin:thm1.1}.  In doing so, we derive consequences for greatest common divisors of pairs of linear recurrence sequences, which are defined over number fields.  

To state our main results, we recall the definition of the generalized logarithmic greatest common divisor of two algebraic numbers \cite[Definition 1.4]{Levin:GCD}.   To begin with, let $M_\kk$ be a set of proper absolute values of a number field $\kk$.  We discuss our normalization conventions for elements of $M_\kk$ in Section \ref{Preliminary}.  

As in \cite{Levin:GCD}, we define the \emph{generalized logarithmic greatest common divisor} of two   algebraic numbers $a,b \in \kk$, not both zero,  to be
$$
\log \gcd(a,b) := -  \sum_{v \in M_\kk} \log^- \max \{|a|_v, |b|_v \},
$$
where $
\log^-(\cdot) := \min \{0,\log(\cdot) \}.
$
This is standard notation.  (Compare with \cite{Silverman:2005} or \cite{Grieve:toric:gcd:2019}, for instance, and the references therein.) 

The following theorem is our first main result.  It is an application of the moving form of Schmidt's Subspace Theorem    \cite[Theorem 1.1]{Ru:Vojta:1997}.  
 
 \begin{theorem}\label{Moving:GCDAffine:Claim}
Let $\kk$ be a number field and $S$ a finite set of places of $\kk$, containing the archimedean places, and let $\Osh_{\kk,S}$ be its ring of $S$-integers.
Let $\Lambda$ be an infinite index set and $u_1,\hdots,u_n \colon \Lambda \rightarrow \Osh^\times_{\kk,S}$ a sequence of  maps. 
Let $f_\alpha, g_\alpha  \in \kk[x_1,\dots,x_n]$ be a collection polynomials with coefficients indexed by $\alpha \in \Lambda$  and with the property that their degrees, $\deg f_{\alpha}$ and $\deg g_{\alpha}$, are positive constants independent of $\alpha\in \Lambda$. 
 Assume that the polynomials $f_\alpha  $ and $g_\alpha $ are coprime and that at least one of them has a nonzero constant term 
for each $\alpha \in \Lambda$.   Furthermore, assume that  
$$\max\{h(f_{\alpha}),\ h(g_{\alpha})\} = \mathrm{o} \left(\max_{1\leq i\leq n} h(u_i(\alpha)) \right) \text{, }$$  
for all $\alpha \in \Lambda$.  

Let $\epsilon > 0$.   In this context, either
 \begin{enumerate}
 \item{ 
  there exists an infinite index subset $A \subseteq \Lambda$  such that 
$$
\log \gcd (f_\alpha(  u_1(\alpha),...,u_n(\alpha)) , g_\alpha(  u_1(\alpha),...,u_n(\alpha)) )  < \epsilon \max_{1\leq i\leq n} h(u_i(\alpha))
$$
  for all  $\alpha \in A$;
or }
\item{
  there exists    a finite union of proper algebraic subgroups $Z$ of $\mathbb G_m^n$ together with a map  
$$\mathbf c: \Lambda \to \kk^\times \text{,}$$ 
with 
$$h(\mathbf c(\alpha))=\mathrm{o} \left(\max_{1\leq i\leq n} h(u_i(\alpha)) \right) \text{, }$$ 
such that 
 $(u_1(\alpha),\hdots,u_n(\alpha))$ is contained in $Z$ translated by the  $\mathbf c(\alpha)$, for  all $\alpha \in \Lambda$}.
 \end{enumerate}

Here, the quantities $h(f_\alpha)$, $h(g_\alpha)$ denote the heights of $f_\alpha, g_\alpha \in \kk[x_1,\dots,x_n]$  whereas $h(u_i(\alpha))$ denotes the height of $u_i(\alpha) \in \Osh_{\kk,S}^\times$.   
\end{theorem} 

The following example was suggested to us by an anonymous   referee. It  
indicates that the existence of an infinite subset $A$ of $\Lambda$ in (i) 
is the best possible  in terms of the cardinality of such $A$. 
 
\begin{example}
Let   $ \Osh_{\mathbb Q,S}^\times\subset \mathbb Q^\times$ be the multiplicative group generated by $\{2,3,-1\}$.  Consider the polynomials $f(x_1,x_2)=x_1-1$ and $g(x_1,x_2)=x_2-1$.  Let  
$\NN = \{1,2,\dots \}$
and define maps 
$$
u_1, u_2:\mathbb N\to  \Osh_{\mathbb Q,S}^\times
$$
by the condition that
$$
u_1(n)=2^n \quad\text{and}\quad  u_2(n)= \begin{cases} 2^n &  \text{ if   $n$ is even}\\ 3^n & \text{ if  $n$  is odd.}
\end{cases}
$$
  Fix some sufficiently small positive real number $\epsilon$,  $0 < \epsilon <1$.   Then, as in \cite[Theorem 1]{Bugeaud:Corvaja:Zannier:2003} and the remarks therein, the inequality 
$$
\log \gcd (f(u_1(n),u_2(n)) , g(u_1(n),u_2(n))) < \epsilon \max\{ h(u_1(n)), h(u_2(n))\}
$$
 is not satisfied for even values of $n$, but holds  for all sufficiently large odd values of $n$.  
On the other hand, for odd values of $n$ the pair $(u_1(n),u_2(n))=(2^n,3^n)$ cannot be contained in a  finite union of proper algebraic subgroups  of $\mathbb G_m^2$ translated by $\mathbf c(n)=\mathrm{o}(n)$. 
 Moreover, for even values of $n$ we have that $u_1(n)/u_2(n)=1$.  For the case that $\Lambda = \NN$, the above discussion relates to the conclusion of Theorem \ref{Moving:GCDAffine:Claim}, in the sense that point (i) holds for all sufficiently large odd values, whereas point (ii) does not hold for all positive integers.  Similarly, when $\Lambda$ is the set of even positive integers, point (ii) is valid whereas point (i) holds for no infinite subset.
 \end{example}
 
As an application of Theorem \ref{Moving:GCDAffine:Claim}, we study the greatest common divisor problem for terms in linear recurrence sequences.
For precise statements, by a \emph{linear recurrence sequence}, we mean a sequence of complex numbers 
 $ \{F(n) \}_{n \in \NN},$
which are defined by
\begin{equation}\label{linear:recurrence:sequence:eqn1}
F(n) := \sum_{i = 1}^r f_i(n) \alpha_i^n \text{,}
\end{equation}
for nonzero polynomials $0 \not = f_i(x) \in \CC[x]$ and nonzero complex numbers $\alpha_i \in \CC^\times$.  

The complex numbers $\alpha_i$, for $i = 1,\dots,r$, are the \emph{roots} of the recurrence sequence.  The sequence \eqref{linear:recurrence:sequence:eqn1} is \emph{non-degenerate} if no $\alpha_i / \alpha_j$ is a root of unity for all $i \not = j$.  It is \emph{algebraic} if $f_i(x) \in \overline{\QQ}[x]$ and $\alpha_i \in \overline{\QQ}^\times$, for all $i = 1,\dots,r$, and \emph{simple} if all of the polynomials $f_i(x)$ are constant.   That the sequence \eqref{linear:recurrence:sequence:eqn1} is defined over $\kk$ means, in particular, that $f_i(x) \in \kk[x]$ and   $\alpha_i \in \kk^\times$,  for all $i = 1,\dots,r$.  

Fix a torsion free multiplicative group $\Gamma \subseteq \CC^\times$, with rank equal to $r$, and let $R_\Gamma$ be the  \emph {ring of linear recurrences with roots belonging to $\Gamma$}.  Recall, that each choice of basis $(\beta_1,\dots,\beta_r)$ for $\Gamma$ allows for the identification
\begin{equation}\label{linear:recurrence:sequence:eqn2}
R_\Gamma \simeq \CC\left[t,x_1^{\pm 1},\dots,x_r^{\pm 1}\right].
\end{equation}
Under this isomorphism \eqref{linear:recurrence:sequence:eqn2}, the linear recurrence \eqref{linear:recurrence:sequence:eqn1}, which is determined by  a Laurent polynomial $f(t,x_1,\dots,x_r) \in R_\Gamma$, is recovered by identifying the variable $x_i$ with the function $n \mapsto \beta_i^n$, for $i=1,\dots r$, and the variable $t$ with the function $n \mapsto n$.
Similarly, in case that $\Gamma$ admits a basis with $\beta_i \in \kk^\times$, for $i = 1,\dots,r$, then we may discuss the ring
$R_\Gamma \simeq \kk[t, x_1^{\pm},\dots,x_r^{\pm}]$ that consists of those   algebraic linear recurrences \eqref{linear:recurrence:sequence:eqn1}   which are defined over $\kk$ and have roots belonging to $\Gamma$.

The following special case of \cite[Theorem 1.11]{Levin:GCD} motivates much of what we do here.

\begin{theorem}[{\cite[Theorem 1.11]{Levin:GCD}}]\label{Levin:thm1.11}
Let $F(n)$ and $G(n)$ be two simple algebraic linear recurrence sequences, defined over $\kk$, and having respective roots $\alpha_i$, $\beta_j$, for $i =1,\dots, s$ and $j = 1,\dots, t$.  Suppose that
$$\max_{i,j} \{ |\alpha_i|_v, |\beta_j|_v \} \geq 1$$  
for any $v\in M_\kk$.
Let $\epsilon>0$.    If the inequality 
$$
 \log  \gcd(  F(m),  G(n)  ) > \epsilon\max\{m,n\}
$$
has infinitely many solutions, then all but finitely many of such solutions must
satisfy one of finitely many linear relations
$$
(m,n)=(a_it+b_i,c_it+d_i), 
$$
for  $t\in\mathbb Z$ and $i=1,\hdots,r$.
Here $a_i, b_i, c_i, d_i \in\mathbb Z$, $a_i, c_i\ne 0$, and the linear recurrences $F(a_i\bullet+b_i)$ and $G(c_i\bullet+d_i)$ have a nontrivial common factor for $i=1,\hdots,r$.
\end{theorem}

In Theorem \ref{Levin:thm1.11}, we use the notations   $ F(a_i\bullet+b_i)$ and $G(c_i\bullet+d_i)$, respectively, to denote the sequences $n \mapsto F(a_in+b_i)$ and $n \mapsto G(c_in+d_i)$.

We again refer to \cite{Levin:GCD} for a survey of related work. 
For the case of nonsimple linear recurrences, Luca, in \cite{Luca:2005}, proved

\begin{theorem}[{\cite[Corollary 3.3]{Luca:2005}}]\label{Luca}
Let $a$ and $b$ be nonzero integers which are multiplicatively independent and let $f_1,f_2,g_1,g_2\in\mathbb Z[x]$ be nonzero polynomials.  Let
$$
F(n) = f_1(n)a^n+f_2(n)
$$
and
$$
G(n) = g_1(n)b^n+g_2(n),
$$
for $n \in \mathbb{N}$.
Then for all  $\epsilon>0$, it holds true that
$$
 \log  \gcd(  F(m),  G(n)  ) < \epsilon\max\{m,n\},
$$
for all but finitely many pairs of positive integers $(m,n)$.
\end{theorem}

Before formulating our main results, in the direction of nonsimple linear recurrence sequences, we  make precise what we mean for a pair of recurrence sequences to be  \emph{separated}.  This notion, which is suitable for our purposes, should be compared with the condition that  a pair of such recurrence sequences be \emph{related}  in the sense of \cite{Schlickewei:Schmiidt:1995} and \cite{Schlickewei:Schmiidt:2000}.

\begin{defn}\label{rootset}
Let 
$F(m) = \sum_{i = 1}^s p_i(m)\alpha_i ^{m}$
and 
$G(n)= \sum_{i = 1}^t q_i(n)\beta_i^n$
be  algebraic linear recurrence sequences   which are defined  over  a number field $\kk$.     Let $\Gamma_F$ and $\Gamma_G$ be, respectively, the multiplicative subgroups of $\kk^\times$ which are generated by their respective roots $\alpha_1,\hdots,\alpha_s$ and $\beta_1,\hdots,\beta_t$.  
We say that  $F$ and $G$ are \emph{separated} if the intersection of $\Gamma_F$ and $\Gamma_G$ is trivial.  Otherwise, we say that they are \emph{not separated}. 
\end{defn}

Theorem \ref{gcdrecurrencebasic} below pertains to the structure of pairs of algebraic linear recurrence sequences which have large greatest common divisor.  This result is stated as follows.

\begin{theorem}\label{gcdrecurrencebasic}
Let 
 $F( m ) = \sum_{i = 1}^s p_i(  m  )\alpha_i^{  m }$ 
and 
$G(n) = \sum_{i = 1}^t q_i(n)\beta_i^n\text{,}$
for $ m,  n \in \mathbb{N}$,
be algebraic  linear recurrence sequences, defined over a number field $\kk$, such that their roots generate together a torsion-free multiplicative subgroup $\Gamma$ of $\kk^{\times}$.  Suppose that  the inequality 
$$\max_{i,j} \{ |\alpha_i|_v, |\beta_j|_v \} \geq 1$$
  is valid for all  $v\in M_\kk$. 
  Let $\epsilon > 0$ and consider 
 the inequality
\begin{align}\label{gcdthm}
 \log  \gcd(  F(m),  G(n)  )  > \epsilon \max\{m,n\},
 \end{align} 
for   pairs of positive integers  $(m,n)\in \NN^2$.   The following two assertions hold true. 
 \begin{enumerate}
 \item   Consider the case that $m = n$.  If the inequality \eqref{gcdthm} is valid for infinitely many positive integers $(n,n)\in \NN^2$, then $F$ and $G$ have a non-trivial common factor in   the ring of linear recurrences  $R_{\Gamma}$.
 \item   Consider the case that $m \not = n$.  If  the inequality  \eqref{gcdthm} is valid for infinitely many pairs  of positive integers   $(m,n) \in \NN^2$, with $m\ne n$, then  
 the linear recurrences $F$ and $G$  are not separated.  Further, in this case,  there exists 
 finitely many
   pairs   of integers $(a,b)\in \ZZ^2$ such that  
$$|ma+nb| =\mathrm{o} (\max\{m,n\}) \text{,}$$
as $m$ or $n$ become sufficiently large.
\end{enumerate}
\end{theorem}
 The conclusion (ii) in Theorem \ref{gcdrecurrencebasic} is illustrated via the following example, which was communicated to us by Aaron Levin.    Consider the case of congruence sequences of the form
$
F(m) = m a^m
$
and
$
G(n) = a^n \text{,}
$
for $a$ some  given positive integer.
Such recurrence sequences have the property that
$
F(m) = G(n)$
with 
$
\text{$m = a^i$ and $n = a^i + i$\text{,} }
$  
for all $i \in \NN$.  
In this case, $F$ and $G$ are not separated and $$|m-n|=i=\mathrm{o} (\max\{m,n\})$$ 
as $i$ becomes sufficiently large.

 As emphasized in \cite[page 434]{Corvaja:Zannier:2002}, if   the multiplicative group $\Gamma$ generated by the roots of $F$ and $G$ has a torsion subgroup, say of order $q$, then the two recurrences $n\mapsto F(qn+ \ell)$ and $n\mapsto G(qn+\ell)$ have roots which generate a torsion-free group, for $0\leq \ell \leq q-1$.   In light of these considerations, Theorem \ref{gcdrecurrence:intro}, below,  is a consequence  of Theorem \ref{gcdrecurrencebasic}.  Among other things, the following Theorem \ref{gcdrecurrence:intro} implies that pairs of not separated algebraic linear recurrence sequences cannot have large greatest common divisor.
 
\begin{theorem}\label{gcdrecurrence:intro}
Let 
$F(  m ) = \sum_{i = 1}^s p_i(  m )\alpha_i ^{  m }$
and 
$G(n)= \sum_{i = 1}^t q_i(n)\beta_i^n$
be  algebraic linear recurrence sequences, which are defined over a  given number field $\kk$.  Suppose that
$$\max_{i,j} \{ |\alpha_i|_v, |\beta_j|_v \} \geq 1$$ 
for any $v\in M_\kk$.
Let $\epsilon>0$.  The following two assertions hold true.
\begin{enumerate}
\item   Consider the case that $m = n$. 
If the inequality
$$
 \log  \gcd(  F(n),  G(n)  ) >\epsilon n,
$$
has infinitely many solutions, then there exists a positive integer $q$ such that all  but finitely many such solutions must be in one of the linear progressions $ q\bullet+b $, $ b=0,\hdots,q-1$, 
and  the corresponding  linear recurrences $F(q\bullet+b)$ and $G(q\bullet+ b)$ have a nontrivial common factor.
\item  Consider the case that $m \not = n$.   If 
  $F$ and $G$   are not separated, then   the inequality
 $$ \log  \gcd(  F(m),  G(n)  )<\epsilon \max\{m,n\}$$
 is valid
 for all but finitely many pairs of positive integers $(m,n)\in\mathbb N^2$.
 \end{enumerate}
 \end{theorem}
 We note that Theorem \ref{Luca} is a direct consequence of Theorem \ref{gcdrecurrence:intro}. 

 We now discuss an  additional 
 application of Theorem \ref{gcdrecurrence:intro}.    To begin with, the Hadamard quotient theorem, conjectured by Pisot and proven by van der Poorten \cite{vdPoorten}, in its simplest form states that ``{\it  if $F(n)$ and $G(n)$ are linear recurrences such that the ratio $F(n)/G(n)$ is an integer for all $n\in\NN$, then $F(n)/G(n)$ is itself a linear recurrence."}
In \cite{Corvaja:Zannier:2002}, Corvaja and Zannier proved the following version (with weaker hypothesis) of this theorem as an application of Schmidt's Subspace Theorem.

\begin{theorem}[{\cite[Theorem 1]{Corvaja:Zannier:2002}}]\label{CZ}
Let $F(n)$ and $G(n)$ be two linear recurrences.  Let $\mathcal{R}$ be a finitely generated subring of $\mathbb{C}$.   If $G(n) \not = 0$ and $F(n) / G(n) \in \mathcal{R}$ for infinitely many $n \in \ZZ_{\geq 0}$, then there exists a polynomial $P(t)$ and positive integers $q, \ell$ with the property that both of the two sequences
$$
\frac{P(n) F(qn+\ell) }{G(qn+\ell) }
\text{
and 
}
\frac{G(qn+\ell)}{P(n)}
$$
are linear recurrences.
\end{theorem} 

Theorem \ref{gcdrecurrence:intro} and some related work imply Proposition \ref{OS} below, which is the fundamental point, in \cite{Corvaja:Zannier:2002}, for the proof of Theorem \ref{CZ}.  Here, we include the case that $m\ne n$ thereby extending \cite[Proposition 2.1]{Corvaja:Zannier:2002}.

\begin{proposition}[{\cite[Proposition 2.1]{Corvaja:Zannier:2002}}]\label{OS}
Let $\kk$ be a number field and $S$ a finite set of places of $\kk$, containing the archimedean places  and having ring of $S$-integers $\Osh_{\kk,S}$.   Let  
$F(m) $ and 
$G(n)$ be   linear recurrence sequences with roots and coefficients  in $\kk$.    Suppose  that the roots of $F$ and $G$ generate together a torsion-free multiplicative subgroup $\Gamma$ of $\kk^\times$.  Suppose furthermore that   $F$ and $G$ are coprime (with respect to $\Gamma$) and that $G$ has more than one root.  Then the following assertions hold true.
 \begin{enumerate}
 \item{   Consider the case that $m=n$. 
There exist at most finitely many  natural numbers $n\in \NN$, for which  
$F(n)/G(n)\in\mathcal O_{\kk,S}$. }
\item{   Consider the case that $m \not = n$. 
There  does not exist  infinitely many pairs   of natural numbers  $(m,n)\in \NN^2$, which have the properties that  $m=\mathrm{o} (n)$ and 
$F(m)/G(n)\in\mathcal O_{\kk,S}$.}
\end{enumerate}
\end{proposition}

The conclusion in Proposition \ref{OS} (ii) follows the suggestion of \cite[page 432]{Corvaja:Zannier:2002}.  
 As mentioned, this article is inspired by recent work of \cite{Levin:GCD} where the primary tool used in the proofs is Schmidt's Subspace Theorem.  Likewise, here, the fundamental aspect to the proof of our results is Schmidt's Subspace Theorem with moving targets, as was  developed by Ru and Vojta in \cite{Ru:Vojta:1997}.  To the best of our knowledge, the results that we obtain here are the first application of this moving form of Schmidt's Subspace Theorem to the study of linear recurrences.   We expect that the point of view taken here may also find similar applications, in more general contexts, that include the study of polynomial and exponential equations.

The relevant background material will be given in the next section.  In Section \ref{sMain}, we  prove Theorem \ref{Moving:GCDAffine:Claim}
by establishing some key lemmas and more technical results. 
In Section \ref{linear:recurrence}, we prove our results which deal with linear recurrence sequences.
 
\section{Preliminaries}\label{Preliminary}

In this section, we fix our notation and recall relevant background material.

\subsection{Heights and Schmidt's Subspace Theorem}
We refer to \cite{Vojta} for more details about this subsection.  Let $\kk$ be a number field and $M_{\kk}$ its set of places.  Our use of the symbol $| \cdot |_v$, for $v \in M_\kk$, is consistent with the use of the symbol $\| \cdot \|_v$ in \cite[Section 2]{Vojta}.  

For example,  given $x \in \kk^\times$, we put
$$
| x |_v := 
\begin{cases}
|\sigma(x)| & \text{if $v \in M_\kk$ is a real place;} \\
|\sigma(x)|^2 & \text{if $v \in M_\kk$ a complex place; and} \\
(\mathcal O_{\kk}:\frak p)^{\ord_{\frak p}(x) }& \text{if $v$ corresponds to a prime ideal $\frak p$ in the ring of integers $\mathcal O_{\kk}$.} 
\end{cases}
$$
Here $\sigma$ denotes, respectively, the real embedding when $v$ is a real place and one of the conjugate pairs of the complex embedding when $v$ is a complex place.  

Recall that, in general, $|\cdot|_v$, for $v \in M_\kk$, is a norm and not an absolute value.  Moreover, for all $x_0,\hdots, x_n$, $a_0,\hdots, a_n\in\kk$ and all $n\in\NN$ it satisfies 
\begin{align}\label{triangle}
|a_0x_0+ \dots +a_n x_n|_v\leq (n+1)^{N_v}\max_{0 \leq i \leq n}\{|x_i|_v \}  \max_{0 \leq i \leq n}\{|a_i|_v \},
\end{align}
where 
$$N_v=
\begin{cases}
1& \text{if $v \in M_\kk$ is a real place;} \\
2 & \text{if $v \in M_\kk$ a complex place; and} \\
0& \text{if $v \in M_\kk$ is a non-archimedean place.} 
\end{cases}
$$

Then, with these notations, these norms satisfy the product formula with multiplicity equal to one
$$
\prod_{v \in M_\kk} |x|_v = 1 \text{,} 
$$
for all $x \in \kk^\times$.
Further, the height of $x \in \kk$ is written as
$$
h(x) := \sum_{v \in M_\kk} \log \max \{1, |x|_v\} 
$$
whereas the height of 
$
\mathbf{x} := [x_0:\dots:x_n] \in \PP^n(\kk)
$
is given by
$$
h(\mathbf{x}) := \sum_{v \in M_\kk} \log \max\{|x_0|_v,\dots,|x_n|_v\}\text{.}
$$
To reduce notation, in what follows, we put:
$$
\| \mathbf{x} \|_v := \max\{|x_0|_v,\dots,|x_n|_v \}.
$$

Similar considerations apply to polynomials
$$
f(x) = \sum_{\mathbf{i}} a_{\mathbf{i}} \mathbf{x}^{\mathbf{i}} \in \kk[x_1,\dots,x_n] \text{.}
$$
Here 
$\mathbf{i} = (i_1,\dots,i_n) \in \ZZ_{\geq 0}^n$ 
and 
$\mathbf{x}^{\mathbf{i}} := x_1^{i_1}\cdot \hdots \cdot x_n^{i_n}\text{.}$  
In particular, the height of $f(x)$ is denoted as
$$
h(f) := \sum_{v \in M_\kk} \log \max_{\mathbf{i}}  \{ | a_{ \mathbf{i}} |_v \} \text{.}
$$
Again, to reduce notation elsewhere, we set
$$
\| f \|_v := \max_{\mathbf{i}} \{ |a_{\mathbf{i}}|_v \}.
$$

Finally, our conventions about Weil functions, for $H \subseteq \PP^n(\kk)$ a hyperplane defined by a linear form
$$
L(x) = a_0x_0+\dots+a_n x_n
$$
are such that
$$
\lambda_{H,v}(\mathbf{x}) := 
\log \left( \frac{ \| \mathbf{x} \|_v \cdot \| L \|_v}{ | a_0 x_0 + \dots + a_n x_n |_v } \right) \text{, }
$$
for 
$\mathbf{x} = [x_0: \dots : x_n ] \in \PP^n(\kk) \setminus H$ 
and $v \in M_\kk$.

We state the following version of Schmidt's Subspace Theorem. (See \cite[Theorem 8.10]{Vojta}.)

\begin{theorem}[Schmidt's Subspace Theorem]\label{Schmidt}  
Let $\kk$ be a number field, let $S$ be a finite set of places of $\kk$ and let $H_1 ,\dots, H_q $ be a collection of distinct hyperplanes in $\PP^n(\kk)$. 
Then for all $\epsilon >0$, the inequality
$$
\sum_{v \in S} \max_J\sum_{j \in J }  \lambda_{H_j,v}(\mathbf{x} ) \leq ( n + 1 + \epsilon ) h(\mathbf{x} ) 
$$
holds true for all $\mathbf{x}\in \PP^n(\kk)$ outside of a finite union of proper linear subspaces.   Here, the maximum is taken over all subsets $J \subseteq \{1,\hdots,q\}$ such that the $H_j$, for $j\in J$, are in general position.
\end{theorem}

\subsection{Fields of moving functions and Schmidt's Subspace Theorem with moving targets}\label{moving:subspace:thm}
For our purposes here, we adopt the moving function formalism of \cite[Section 1]{Ru:Vojta:1997}.  Let $\Lambda$ be an infinite index set and   fix an infinite subset $A \subseteq \Lambda$.   We define $\mathcal{R}^0_A$ to be the set of equivalence classes of pairs $(C,a)$, where $C \subseteq A$ is a subset with finite complement and where $a \colon C \rightarrow \kk$ is a map.  We say that two such pairs are equivalent, written $(C,a) \sim (C', a')$, if there exists a subset $C'' \subseteq C' \bigcap C$ that has finite complement in $A$ and such that the restrictions of $a$ and $a'$ to $C''$ coincide.

We now recall the field of moving functions associated with a set of moving hyperplanes introduced in \cite[Definition 1.2]{Ru:Vojta:1997}. 
A \emph{moving hyperplane}, indexed by $\Lambda$ over $\kk$, is a map 
$H:\Lambda\to\PP^n(\kk)^*\text{,}$ 
which is defined by 
$\alpha\mapsto H(\alpha)\text{.}$

Given a collection, $\mathcal H$, of moving hyperplanes 
$H_i(\alpha) \subset \PP^n(\kk)\text{,}$ 
for each $\alpha \in \Lambda$ and all $i = 1,\dots,q$, choose  $a_{i,0}(\alpha),\hdots,  a_{i,n}(\alpha) \in \kk$,  not all zero, and such that $H_i(\alpha)$ is the hyperplane determined by the equation  
\begin{equation}\label{movinghyperplane}
a_{i,0}(\alpha)x_0+\hdots+ a_{i,n}(\alpha)x_n=0 \text{,} 
\end{equation}
for  $i = 1,\dots,q$.
In this way, $\mathcal{H}$ determines a sequence of maps 
\begin{equation}\label{moving:maps:defn}
\mathbf{a} = \{a_{i,j} \colon \Lambda \rightarrow \kk \}_{1 \leq i \leq q \text{ and }  0 \leq j  \leq n }\text{.}
\end{equation}

In what follows, we require a concept of \emph{coherence} for infinite subsets $A \subseteq \Lambda$ with respect to a collection of moving hyperplanes $\mathcal{H}$.

\begin{defn}[{\cite[Definition 1.1]{Ru:Vojta:1997}}]\label{moving:maps:defn:1}
An infinite subset $A \subseteq \Lambda$ is said to be \emph{coherent} with respect to $\mathcal H$, 
or with respect to the collection of maps \eqref{moving:maps:defn}, if, for each block homogeneous polynomial 
$$P(\mathbf{x}) \in \kk[x_{1,0},\dots,x_{1,n}, \dots, x_{q,0},\dots,x_{q,n}]\text{,}$$ 
either $P(\mathbf{a}(\alpha)) = 0$, for all $\alpha \in A$; or $P(\mathbf{a}(\alpha)) = 0$, for at most finitely many $\alpha \in A$.  Here, we have put
$$
\mathbf{a}(\alpha) = (a_{1,0}(\alpha),\dots, a_{1,n}(\alpha),\dots, a_{q,0}(\alpha),\dots, a_{q,n}(\alpha)).
$$
\end{defn}

 In our present setting, we obtain a field of moving functions in the following way.

\begin{defn}[{\cite[Definition 1.2]{Ru:Vojta:1997}}]\label{moving:maps:example:2}
Let $A \subseteq \Lambda$ be an infinite subset which is coherent with respect to $\mathcal H$, or, equivalently, with respect to the  
 collection of maps \eqref{moving:maps:defn}.  
 We embed $\kk$ into $\mathcal{R}^0_A$ as constant functions.  For each $i\in\{ 1,\dots, q\}$ and each $\alpha\in A$, there exists $\nu\in\{0,\dots,n\}$ such that $a_{i,\nu}(\alpha)\ne 0$.  Therefore, we can find $\nu\in\{0,\dots,n\}$ such that $a_{i,\nu}(\alpha)\ne 0$ for infinitely many $\alpha\in A$.  Moreover, $a_{i,\nu}(\alpha)\ne 0$ for all but finitely many $\alpha\in A$ since $A$ is coherent with respect to $\mathcal H$.  We will assume that $a_{i,\nu}(\alpha)\ne 0$, for all $\alpha\in A$, by replacing $A$ by a subset with finite complement, which is still coherent with respect to $\mathcal H$.  Then $a_{i,\mu} / a_{i,\nu}$ defines an element of $\mathcal{R}^0_A$.  Moreover, by coherence, the subring of $\mathcal{R}^0_A$  generated by all such elements $a_{i,\mu} / a_{i,\nu}$ is an integral domain, which we denote by $\mathcal{R}_A$.  In this context, the field of fractions of $\mathcal{R}_A$, denoted by   $\mathcal{K}_{\mathcal H,A}$,    is \emph{the field of moving functions} for $\mathcal{H}$ with respect to $A$.
\end{defn}

Before proceeding further, we make a handful of remarks about this construction of fields of moving functions.

\begin{remark}\label{moving:hypersurface:remark}  The following three assertions hold true.
\begin{enumerate}
\item[(i)]{The field $\mathcal{K}_{\mathcal H,A}$ is independent of the choice of coefficients of the linear forms.
}
\item[(ii)]{
The existence of infinite subsets $A \subseteq \Lambda$ which are coherent, in the sense of Definition \ref{moving:maps:defn:2}, follows as in \cite[Lemma 1.1]{Ru:Vojta:1997}. 
}
\item[(iii)]{
Given infinite subsets $B \subseteq A \subseteq \Lambda$, if $A$ is coherent, then so is $B$ and $\mathcal{K}_{\mathcal H,B} \subseteq \mathcal{K}_{\mathcal H,A}$.
}
\end{enumerate}
\end{remark}

More generally, similar to \cite[Definition 1.2]{Chen:Ru:Yan:2015}, we may formulate    the  concept of a  collection of moving polynomials, of given arbitrary inhomogeneous degrees,  
indexed by $\Lambda$  together with a concept of coherence.   Such notions are important for our purposes here.

\begin{defn}\label{moving:maps:defn:2}
Let $\Lambda$ be an infinite index set,  and let $f_i $, for $1\leq i\leq q$, be a collection of moving polynomials indexed by $\Lambda$, of  degree  $d_i$, for $1\leq i\leq q$.  In particular, it holds true that 
$$
f_i(\alpha)=\sum_{\mathbf{i}\in\mathcal I_{d_i}}a_{i,\mathbf{i}}(\alpha){\bf x}^{\mathbf{i}}\in\kk[x_1,\hdots,x_{n_i}],
$$
where $\mathcal I_{d_i}$ is the set containing all monomials in $x_1,\hdots,x_{n_i}$ of degree no bigger than $d_i$.   We may decompose  the index sets   $\mathcal{I}_{d_i}$ as   
$\mathcal I_{d_i}=\{I_{j,1},\hdots,I_{j,n_{d_i}}\}\text{,}$
for $1 \leq i \leq q$.
In this context, we say that an infinite subset $A \subseteq \Lambda$ is \emph{coherent} with respect to   $f_1,\hdots,f_q$, 
  if, for each  polynomial 
$$P(\mathbf{x}) \in \kk\left[x_{1,I_{1,1}},\dots,x_{1,I_{1,n_{d_1}}}, \dots, x_{q,I_{q,1}},\dots,x_{q,I_{q,n_{d_q}}}\right]\text{,}$$ 
either $P(\mathbf{a}(\alpha)) = 0$, for all $\alpha \in A$; or $P(\mathbf{a}(\alpha)) = 0$, for at most finitely many $\alpha \in A$.  Here, we have put
$$
\mathbf{a}(\alpha) = \left(a_{1,I_{1,1}}(\alpha),\dots, a_{1,I_{1,n_{d_1}}}(\alpha),\dots, a_{q,I_{q,1}}(\alpha),\dots, a_{q,I_{q,n_{d_q}}}(\alpha)\right).
$$
\end{defn}

 At times, via a Veronese embedding as in Example \ref{moving:hypersurfaces:defn:3} below, it is useful to view a collection of moving hypersurfaces as a collection of moving hyperplanes.    Such collections of moving hypersurfaces determine moving fields of functions in the following sense.

\begin{defn}\label{moving:hypersurfaces:defn:3} 
Let $\Lambda$ be an infinite index set, let $\mathcal D$ be a set of moving hypersurfaces, of degree $d$ indexed by $\Lambda$, and let $F_1,\hdots,F_q$ be defining homogeneous degree $d$ polynomials, which correspond to these moving hypersurfaces.  Via a Veronese embedding, we view each of these moving forms $F_1(\alpha),\hdots,F_q(\alpha)$, for $\alpha\in \Lambda$, as hyperplanes in $\PP^{\binom{n+d}{d}-1 }({\kk})$.  Let $\mathcal H_{\mathcal D}$ be the set of these hyperplanes and fix  $A \subseteq \Lambda$, an infinite subset which is coherent with respect to $\mathcal H_{\mathcal D}$.  Then we have a   \emph{moving field} $\mathcal{K}_{\mathcal H_{\mathcal D},A}$ associated to $\mathcal D$.
\end{defn}

We also require a concept of \emph{moving points} that are \emph{nondegenerate} with respect to a collection of moving hyperplanes.

\begin{defn}[{\cite[Definition 1.3]{Ru:Vojta:1997}}]\label{moving:maps:defn:3}
Let 
\begin{equation}\label{moving:maps}
x_i \colon \Lambda \rightarrow \kk
\end{equation}
be a collection of maps, for $i = 0,\dots, n$, with the property that for all $\alpha \in \Lambda$, at least one $x_i(\alpha) \not = 0$.   Such maps define moving points
\begin{equation}\label{moving:points}
\mathbf{x}(\alpha) = [x_0(\alpha):\dots:x_{n}(\alpha)] \in \PP^n(\kk)\text{,}
\end{equation}
for each $\alpha \in \Lambda$.
In this context, we say that the moving points \eqref{moving:points} are \emph{nondegenerate} with respect to a finite collection $\mathcal H$ of moving hyperplanes  if for each infinite coherent subset $A \subseteq \Lambda$, 
the restrictions of all $x_i$ to $A$ are linearly independent over  $\mathcal{K}_{\mathcal H,A}$.  We say that $\mathbf{x}$ is \emph{degenerate}, with respect to $\mathcal{H}$, in case that it is not \emph{nondegenerate}.
\end{defn}

Recall that the following form of Schmidt's Subspace Theorem, with moving targets, was obtained by Ru-Vojta in \cite{Ru:Vojta:1997}.  It was then extended further by Chen-Ru-Yan in \cite{Chen:Ru:Yan:2015}.  We use this result in our proof of Theorem \ref{Moving:GCDAffine:Claim}. 

\begin{theorem}[{\cite[Theorem 1.1]{Ru:Vojta:1997}}, {\cite[Theorem D]{Chen:Ru:Yan:2015}}]\label{Moving:Schmidt:Subspace:Intro}
Let $\kk$ be a number field, $S$ a finite set of places of $\kk$, $\Lambda$  an infinite index set, let $\mathcal H=\{H_1,\dots, H_q\}$ be a collection of moving hyperplanes in $\PP^n$, indexed by  $ \Lambda$   and defined  over $\kk$, and let $\mathbf{x} \colon \Lambda \rightarrow \PP^n(\kk)$ be a collection of moving points such that
\begin{enumerate}
\item[{\rm (i)}]{$\mathbf{x} \colon \Lambda \rightarrow \PP^n(\kk)$ is nondegenerate with respect to $\mathcal H$; and}
\item[{\rm (ii)}]{ $h (H_j(\alpha)) =  \mathrm{o}(h (\mathbf{x}(\alpha)))$, for all  $\alpha\in\Lambda$ and all $j = 1,\dots, q$.}
\end{enumerate}
Then, for each $\epsilon > 0$, there exists an infinite index subset $A \subseteq \Lambda$ such that the inequality
$$
\sum_{v \in S} \max_J\sum_{j \in J }  \lambda_{H_j(\alpha),v}(\mathbf{x}(\alpha)) \leq ( n + 1 + \epsilon ) h (\mathbf{x}(\alpha))
$$
holds true for all $\alpha \in A$.  Here, the maximum is taken over all subsets 
$J \subseteq \{1,\hdots,q\}$ 
such that the $H_j(\alpha)$, for $j\in J$,   and all $\alpha \in \Lambda$, are linearly independent.   
\end{theorem}

 
\section{Proof of  Theorem \ref{Moving:GCDAffine:Claim}}
\label{sMain}

It is convenient to use the term \emph{slow growth} in the following situation.
 Fix a collection of moving polynomials 
 $f_\alpha(x) \in \kk[x_1,\dots,x_n]\text{,}$ 
 with coefficients indexed by $\alpha \in \Lambda$, together with a sequence of maps 
 $u_i \colon \Lambda \rightarrow \Osh_{\kk,S}^\times\text{,}$ 
 for $i = 1,\dots,n$. We say that  these polynomials  $f_\alpha $ have 
   \emph{slow growth} with respect to the moving points
 $$\mathbf{u}(\alpha) := (u_1(\alpha),\dots, u_n(\alpha)) \in \G^n_m(\Osh_{\kk,S})$$   
 in case that
 \begin{equation}\label{slow:growth:eqn1}
 h(f_\alpha ) = \mathrm{o}\left( \max_{1 \leq i \leq n} h(u_i(\alpha)) \right) \text{, }
 \end{equation}
 for each $\alpha \in \Lambda$.

\subsection{Proof of Theorem \ref{Moving:GCDAffine:Claim} by two key theorems}
Theorem \ref{Moving:GCDAffine:Claim} is a consequence of Theorems \ref{Moving:GCDAffine:Claim:1} and \ref{proximitygcd} below.   First, we state Theorem \ref{Moving:GCDAffine:Claim:1}.  It is the moving target analogue of \cite[Theorem 3.2]{Levin:GCD}.

\begin{theorem}\label{Moving:GCDAffine:Claim:1}
Let $\kk$ be a number field and $S$ a finite set of places of $\kk$, containing the archimedean places, and let $\Osh_{\kk,S}$ be its ring of $S$-integers.
Let   
$u_1,\hdots,u_n \colon \Lambda \rightarrow \Osh^\times_{\kk,S}$ 
be a sequence of maps.
Let $f_\alpha(x)$ and $g_\alpha(x)$  be  coprime  moving polynomials in $ \kk[x_1,\hdots,x_n]$ indexed by  a fixed infinite index set $\Lambda$ and having the properties that their degrees,  $\deg f_{\alpha}$ and $\deg g_{\alpha}$, are positive constants independent of $\alpha\in \Lambda$.
Furthermore, assume that these moving polynomials $f_\alpha $ and $g_\alpha $  have slow growth with respect to  the moving points
$$\mathbf{u}(\alpha):= (u_1(\alpha),\dots, u_n(\alpha)) \in \mathbb{G}_m^n(\Osh_{\kk,S}) \text{,} $$
for all $\alpha \in \Lambda$.   If $\epsilon > 0$, then either 
 \begin{enumerate}
 \item{ 
  there exists an infinite index subset $A \subseteq \Lambda$  such that
$$
-\sum_{v \in M_\kk \setminus S} \log^- 
\max \{
|f_\alpha(\mathbf{u}(\alpha)) |_v, |g_\alpha(\mathbf{u}(\alpha)) |_v \}
 < \epsilon \max_{1\leq i\leq n} h(u_i(\alpha))
$$
for all $\alpha \in A$; or
}
\item{
 there exists   a   finite union of  proper algebraic subgroups $Z$ of $\mathbb G_m^n$ together with a map 
$$\mathbf c: \Lambda \to \kk^\times \text{,}$$ 
 
with 
$$h(\mathbf c(\alpha))=\mathrm{o} \left(\max_{1\leq i\leq n} h(u_i(\alpha)) \right) \text{, }$$ 
such that 
 $(u_1(\alpha),\hdots,u_n(\alpha))$ is contained in $Z$ translated by   the  $\mathbf c(\alpha)$, for  all $\alpha \in \Lambda$}.
  \end{enumerate}
\end{theorem}

Theorem \ref{proximitygcd} is formulated in the following way.
It is the moving target form of \cite[Theorem 3.3]{Levin:GCD}.

\begin{theorem}\label{proximitygcd}
Let $\kk$ be a number field, $S$ a finite set of places of $\kk$, containing the archimedean places, and $\Osh_{\kk,S}$ the ring of $S$-integers.
Let   $u_1,\hdots,u_n \colon \Lambda \rightarrow \Osh^\times_{\kk,S}$ 
be a sequence of maps.  Let $f_\alpha(x)$ be  polynomials in $ \kk[x_1,\hdots,x_n]$ with coefficients indexed by $\Lambda$ such that  $\deg f_{\alpha}$ is a positive constant, independent of $\alpha\in \Lambda$, and such that $f_{\alpha}$ does not vanish at the origin for every $\alpha\in\Lambda$.
Assume that the $f_\alpha $ have slow growth, with respect to   $\mathbf{u}(\alpha):= (u_1(\alpha),\dots, u_n(\alpha)) \in \mathbb{G}_m^n(\Osh_{\kk,S})$, for  all $\alpha \in \Lambda$.    Then for all 
 $\epsilon > 0$, either 
 \begin{enumerate}
 \item{  
  there exists an infinite index subset $A \subseteq \Lambda$ such that 
$$
-\sum_{v \in S} \log^- |f_\alpha(\mathbf{u}(\alpha)) |_v  < \epsilon \max_{1\leq i\leq n} h(u_i(\alpha))
$$
  for all  
$\alpha \in A$; or
}
\item{
 there exists a   finite union of proper  proper algebraic subgroup $Z$ of $\mathbb G_m^n$  together with a map $$\mathbf c: \Lambda \to \kk^\times \text{,}$$ 
with 
$$h(\mathbf c(\alpha))=\mathrm{o} \left(\max_{1\leq i\leq n} h(u_i(\alpha)) \right) \text{, }$$ 
such that 
 $(u_1(\alpha),\hdots,u_n(\alpha))$ is contained in $Z$ translated by  the  $\mathbf c(\alpha)$, \ for  all $\alpha \in \Lambda$}.
\end{enumerate}
 \end{theorem}

We now prove Theorem \ref{Moving:GCDAffine:Claim} assuming Theorems \ref{Moving:GCDAffine:Claim:1} and \ref{proximitygcd}. 
 
\begin{proof}[Proof of Theorem \ref{Moving:GCDAffine:Claim}] 
Suppose that the conclusion of (ii), in Theorem \ref{Moving:GCDAffine:Claim}, does not hold.  
Then by statement (i) of Theorem \ref{Moving:GCDAffine:Claim:1},   applied to the case that $\epsilon' = \epsilon / 2 > 0$,
 there exists an infinite index subset $A \subseteq \Lambda$   such that
\begin{equation}\label{proof:thmMoving:GCDAffine:Claim:eqn1}
-\sum_{v \in M_\kk \setminus S} \log^- 
\max \{
|f_\alpha(\mathbf{u}(\alpha)) |_v, |g_\alpha(\mathbf{u}(\alpha))|_v \}
 < \frac{\epsilon}{2} \max_{1\leq i\leq n} h(u_i(\alpha))\text{, }
\end{equation}
for all  $\alpha \in A$.  

Since, for each $\alpha\in \Lambda$, the polynomials $f_\alpha(x)$ and $g_\alpha(x)$ do not both vanish at $(0,\hdots,0)$, 
without loss of generality, we may assume that there exists an infinite subset $A'$ of $A$  such that $f_\alpha(x)$ does not vanish at $(0,\hdots,0)$, for each $\alpha\in A'$.  
We now apply Theorem \ref{proximitygcd} to $f_\alpha(x)$, with $\alpha\in A'$ and $ \epsilon / 2 >0$.  Again, we deduce from statement (i) of Theorem \ref{proximitygcd}  that   
   there exists an infinite index subset $A'' \subseteq A'$, which has the property that
\begin{align}\label{proof:thmMoving:GCDAffine:Claim:eqn2}
-\sum_{v \in  S} \log^- 
\max \{
|f_\alpha(\mathbf{u}(\alpha)) |_v, |g_\alpha(\mathbf{u}(\alpha)) |_v \}   \leq -\sum_{v \in S} \log^- |f_\alpha(\mathbf{u}(\alpha)) |_v   < \frac{\epsilon}{2} \max_{1\leq i\leq n} h(u_i(\alpha))
\end{align}  
  for all $\alpha \in A'' $.
 The conclusion (i), desired by Theorem \ref{Moving:GCDAffine:Claim}, then follows by combining \eqref{proof:thmMoving:GCDAffine:Claim:eqn1} and \eqref{proof:thmMoving:GCDAffine:Claim:eqn2}.
\end{proof}

\subsection{Two lemmas with moving targets}
Our goal here, is to establish a moving target version of a result of Laurent. (Compare with \cite[Lemma 6]{Laurent} or \cite[Theorem 2.1]{Levin:GCD}.)  It can be viewed as an analogue  of the Borel Lemma with moving targets.
  (See   \cite[Lemma 12]{Guo2} or \cite[Lemma 5.5]{Levin:Wang:GCD}.) 
 We refer to  \cite[Theorem A.3.3.2]{ru2001nevanlinna} for the case of constant coefficients.
 
We will apply this lemma, which we state as Lemma \ref{Mborel1}, in several places.    For our purposes,  it replaces the Skolem-Mahler-Lech Theorem in the proof of Theorem \ref{Levin:thm1.11} (\cite[Theorem 1.11]{Levin:GCD}).
Indeed, it can be used to establish the  Skolem-Mahler-Lech Theorem.  
Note that Lemma \ref{Mborel1} is a consequence of the classical non-moving version of Schmidt's Subspace Theorem (Theorem \ref{Schmidt}).  
\begin{lemma}\label{Mborel1}  
Let $\kk$ be a number field, let $S$ be a finite set of places of $\kk$,  containing the archimedean places and with rings of $S$-integers $\mathcal{O}_{\kk,S}$. Let $\Lambda$ be an infinite index set, let  
$$c_i\colon \Lambda \rightarrow \kk^\times$$
be maps, for $0\leq i\leq n$, and let
$\mathcal H$ be the moving  hyperplane in $\mathbb P^n$ defined by 
$$c_0(\alpha)x_0 + \hdots+c_n(\alpha)x_n= 0 \text{,}$$
for $\alpha\in \Lambda$.  
Let $A \subseteq \Lambda$ be an infinite subset which is coherent with respect to $\mathcal H$. 
Let  
$u_0,\hdots,u_n \colon \Lambda \rightarrow \Osh^\times_{\kk,S}$
be a sequence of maps and put 
${\bf u}=[u_0:\cdots:u_n]  \colon \Lambda \rightarrow \mathbb P^n$. 
Assume that    
\begin{align}\label{heightcond}
 h(c_i(\alpha)) = \mathrm{o}( h({\bf u}(\alpha)))\text{,} 
 \end{align} 
for $0\leq i\leq n$, and 
suppose that
\begin{align}\label{unitrelation}
c_0(\alpha)u_0(\alpha) + \hdots + c_n(\alpha)u_n(\alpha)= 0 
\end{align}
for all $\alpha\in A$.
Then,  for each  $i$, $0 \leq   i \leq n$,  there exists $j$,  with $0 \leq  j\leq n$ and $i \not = j$, such that 
$$h(u_i(\alpha)/u_j(\alpha))= \mathrm{o}( h({\bf u}(\alpha)))\text{, }$$ 
 for infinitely many $\alpha \in A$.
\end{lemma}
\begin{proof}
 Let $\mathcal{K}_{\mathcal H,A}$ be the moving field as defined in  Definition  \ref{moving:maps:example:2}.   As $c_0(\alpha)\ne 0$ for all $\alpha\in\Lambda$, we may assume that $c_0(\alpha)=1$, for all $\alpha \in A$, by  dividing the maps $c_i$,    for  $i = 0,\dots,n$,  by $c_0$ without changing the assumption \eqref{heightcond} on the height of $c_i(\alpha)$.
Then $c_i\in  \mathcal{K}_{\mathcal H,A}^\times$, for each  $0 \leq i \leq n$,
 and $\mathcal{K}_{\mathcal H,A}$ has slow growth with respect to ${\bf u}$.
Consequently,  the equation \eqref{unitrelation} implies that each $u_i$ is a $\mathcal{K}_{\mathcal H,A}$-linear combination of $u_0,\hdots,u_{i-1},u_{i+1},\hdots,u_n$. By reindexing the maps $u_i$ if necessary, it suffices to show the statement for $i=0$.

Under these assumptions, from \eqref{unitrelation} and by reindexing the maps $u_j$, $1\leq j\leq n$, we arrive at the $\mathcal{K}_{\mathcal H,A}$-linear relation
\begin{equation}\label{borel:lemma:eqn1}
u_0= a_1 u_1 + \hdots + a_m u_m,
\end{equation}
where $a_j\in \mathcal{K}^{\times}_{\mathcal H,A}$  ,  for each $1\leq j\leq m\leq n$,  and where  
$u_1,\hdots,u_m$ are  $\mathcal{K}_{\mathcal H,A}$-linearly independent.
 Since $A$ is coherent, each nonzero  $a_i$ has finitely many zeros in $A$.   
 Therefore, there exists a subset $A'$ of $A$,  with finite complement, such that
\begin{equation}\label{borel:lemma:eqn2}
u_{0}(\alpha) = a_1(\alpha)u_1(\alpha) + \hdots + a_m(\alpha)u_m(\alpha)\text{, }
\end{equation}
and $a_j(\alpha) \not = 0$, $1 \leq j \leq m$, for $\alpha\in A'$.
 
If $m =1$, then we identify $u_0/u_1$ with an element of $\mathcal{K}_{\mathcal H,{  A'}}{  \subseteq \mathcal K_{\mathcal H,A}}$.  The assertion is then clear since $\mathcal K_{\mathcal H,A}$ is a moving field of functions which has slow growth with respect to the maps ${\bf u}$.  Thus, henceforth, we may assume that $m \geq 2$.

Now consider the collection of moving points
\begin{equation}\label{borel:lemma:eqn3}
\mathbf{y}(\alpha) := [a_1(\alpha) u_1(\alpha) : \dots : a_m(\alpha) u_m(\alpha) ] \in \PP^{m-1}(\kk) \text{,}
\end{equation}
which are indexed by $\alpha \in A'$.  We then apply Theorem \ref{Schmidt}, the classical (non-moving) version of Schmidt's Subspace Theorem, with respect to the coordinate hyperplanes 
$H_{j-1}:=\{x_{j-1}=0\}\text{,}$
 for $j = 1,\dots, m$, and the  diagonal   hyperplane 
$H_m := \{x_0 + \hdots + x_{m-1}=0\}\text{.}$  

Put $\epsilon = 1/2$.  Our conclusion, then, is that there exists a Zariski closed subset $Z \subsetneq \PP^{m-1}(\kk)$, which is a union of finitely many hyperplanes in $\PP^{m-1}(\kk)$, such that if $y(\alpha) \not \in Z$, then
\begin{equation}\label{borel:lemma:eqn4}
\sum_{v \in S} \sum_{i = 0}^{m} \lambda_{H_i,v}(\mathbf{y}(\alpha)) \leq \left(m + \frac{1}{2} \right) h(\mathbf{y}(\alpha)).
\end{equation}
  As elements  of  $\kk$ are identified with constant functions in $\mathcal K_{\mathcal H,A}$, the $\mathcal K_{\mathcal H,A}$-linearly independent assumption on  the $u_1,\hdots,u_{m}$ implies that there exists an infinite subset $A''$ of $A'$ such that 
$\mathbf{y}(\alpha)\not\in  Z$ for $\alpha\in A''$.   (Here, we have used  the fact that $Z$ is a finite union of hyperplanes.)  
Therefore \eqref{borel:lemma:eqn4} holds for all $\alpha\in A''$.

On the other hand,  the definition of the local Weil functions 
and the product formula imply that
 \begin{align}\label{borel:lemma:eqn5}
\sum_{v \in M_{\kk} \setminus S} \sum_{i = 0}^m \lambda_{H_i,v}(\mathbf{y}(\alpha)) + \sum_{v \in S} \sum_{i = 0}^m \lambda_{H_i,v}(\mathbf{y}(\alpha)) =  (m+1) h(\mathbf{y}(\alpha)) + h([a_1(\alpha):\cdots : a_m(\alpha)]) \text{.}
\end{align}
Moreover, since each of the $u_i(\alpha)$ are $S$-units,

\begin{align}\label{borel:lemma:eqn5'}
\begin{split}
\sum_{v \in M_{\kk} \setminus S} \sum_{i =0}^{ m-1} \lambda_{H_i,v}(\mathbf{y}(\alpha))
&=\sum_{v \in M_{\kk} \setminus S}\log \max\{|a_1(\alpha)|_v, \hdots, |a_m(\alpha)|_v\}  \cr
&\geq h( a_1(\alpha),\hdots , a_m(\alpha))- \sum_{v \in S}\sum_{j=1}^m\log^+ \left( |a_j(\alpha)|_v \right) \cr
&\geq h( a_1(\alpha),\hdots , a_m(\alpha))-\sum_{j=1}^m h(a_j(\alpha)).
\end{split}
\end{align}
 By combining \eqref{borel:lemma:eqn4}, \eqref{borel:lemma:eqn5} and \eqref{borel:lemma:eqn5'}, we obtain
\begin{equation}\label{borel:lemma:eqn6}
\frac{1}{2} h(\mathbf{y}(\alpha)) \leq  \sum_{j=1}^m h(a_j(\alpha))\leq \mathrm{o}(h(\mathbf{u}(\alpha)))
\end{equation}
for all $\alpha \in A''$.  
 Indeed, this equation \eqref{borel:lemma:eqn6} follows because each of the $a_i$ are in $\mathcal{K}_{\mathcal H,A}$ and because $\mathcal{K}_{\mathcal H,A}$ has slow growth with respect to ${\bf u}$.
Finally, from \eqref{borel:lemma:eqn2}, we have
 \begin{align}\label{borel:lemma:eqn7}
\begin{split} 
h\left(\frac{u_0(\alpha)}{u_j(\alpha)}\right )   
&
 \leq h(a_1(\alpha),\hdots , a_m(\alpha))+h\left({ \frac{u_{1}(\alpha)}
 {u_j(\alpha)},\hdots ,\frac{u_{m}(\alpha)}{u_j(\alpha)},1}\right)+  \mathrm{O}(1)  
\\
& 
\leq  h(\mathbf{y}(\alpha))+\mathrm{o}(h(\mathbf{u}(\alpha))).
\end{split}
\end{align}
Then our assertion is valid, by \eqref{borel:lemma:eqn6} and \eqref{borel:lemma:eqn7}, for all $\alpha \in A''$,   which is an infinite subset of $A$. 
\end{proof}

 We mention one other lemma which we require.

\begin{lemma}\label{linearproximitygcd}
Let $\kk$ be a number field, $S$ a finite set of places of $\kk$ containing the archimedean places   and with ring of $S$-integers $\mathcal{O}_{\kk,S}$.   Let  
$u_0,\hdots,u_n \colon \Lambda \rightarrow \Osh^\times_{\kk,S}$
be a sequence of maps and put 
$\mathbf{u}=[u_0:\cdots:u_n].$  
Let   $H_{\alpha} \subset \mathbb{P}^n$ be a collection of moving  hyperplanes defined by linear forms 
$L_\alpha(x) \in \kk[x_0,\hdots,x_n]$ 
and with coefficients indexed by $\alpha \in \Lambda$.  Assume that    
$$h(L_{\alpha}) = \mathrm{o}(h({\bf u}(\alpha))),$$
for all $\alpha \in \Lambda$.
Let $\epsilon > 0$.   Then  either 
 \begin{enumerate}
 \item{
  there exists an infinite index subset $A \subseteq \Lambda$ such that
 $$
 \sum_{v \in S} \lambda_{H_{\alpha},v}({\bf u}(\alpha))<  \epsilon h({\bf u}(\alpha))
$$ 
 for all $\alpha \in A$; or
 }
\item{ 
 there exists an infinite index subset $A \subseteq \Lambda$ and  indices  $i$ and $j$, with $0\leq i\ne j\leq n$,
such that 
$$h(u_i(\alpha)/u_j(\alpha)) = \mathrm{o}(h({\bf u}(\alpha)))\text{,}$$  
 for all  $\alpha \in A$. 
 }
 \end{enumerate}
\end{lemma}

\begin{proof} 
By rearranging the index set in some order, if necessary, we may write  
\begin{align}\label{L}
L_\alpha(u_1,\hdots,u_n)= \sum_{j=0}^{\ell}a_{j}(\alpha)u_j(\alpha),
\end{align}      
where, 
$a_{j}(\alpha)\neq 0$, for $0\leq j\leq \ell\leq n$ and infinitely many  $\alpha\in\Lambda$.  Replacing $\Lambda$ with an infinite subset if necessary, we may assume that 
$a_{j}(\alpha)\neq 0$, for all $\alpha\in\Lambda$ and all $0\leq j\leq \ell$.

Set
$
\frak u :=(u_0,\hdots,u_\ell).
$ 
By evaluation at $\alpha \in \Lambda $, $\frak u$ determines a collection of moving points in $\PP^\ell$.  Let $H_i \subset \PP^\ell$, for $0\leq i\leq \ell$, be the coordinate hyperplanes and $H_{\ell+1}$ the moving hyperplane $H$ defined by  $L_\alpha$ in \eqref{L}.   By construction, the set of $\ell+2$ hyperplanes $H_0(\alpha),\hdots,H_{\ell+1}(\alpha)$ are in general position for all $\alpha\in \Lambda$. 
 
If  $\frak u$ is  degenerate with respect to the moving hyperplanes  $H_i$, for $0\leq i\leq \ell+1$, then we use Lemma \ref{Mborel1} to deduce the conclusion given by part (ii) of Lemma \ref{linearproximitygcd}.
 
Suppose now, that  $\frak u$
is nondegenerate
with respect to the moving hyperplanes  $H_{i}$, for $0\leq i\leq \ell+1$.  
By  Theorem \ref{Moving:Schmidt:Subspace:Intro},  the moving form of Schmidt's Subspace Theorem,
there exists an infinite index set  $A\subseteq \Lambda$ 
 such that 
\begin{align}\label{proximityHt0}
\sum_{i=0}^{\ell+1} \sum_{v\in S}\lambda_{H_{i} (\alpha),v}(\frak u(\alpha)) < (\ell+1 + \epsilon) h(\frak u(\alpha))
\end{align}
for all  $\alpha \in A$.
By assumption, 
$u_i(\alpha) \in \Osh_{\kk,S}^\times\text{,}$ 
for all $1\leq i\leq n$.  We then have the relation
$$\sum_{v\in S}\lambda_{H_{i} (\alpha),v}(\frak u(\alpha)) = h(\frak u(\alpha))\text{,}$$
for each fixed $i = 0,\dots, \ell$  and all  $\alpha\in  A$.
We can now derive from \eqref{proximityHt0} the inequality
\begin{align}
\sum_{v\in S}\lambda_{H_{\alpha},v}(\frak u(\alpha))< \epsilon h(\frak u(\alpha))\leq \epsilon h({\bf u}(\alpha)) \text{, }
\end{align}
for all 
$\alpha \in A.$
This concludes the proof.
 \end{proof}

In most cases, Lemmas \ref{Mborel1} and \ref{linearproximitygcd} will be applied to linear relations amongst monomials in the maps $u_i$.
We make a convenient statement of the implication of Lemma \ref{Mborel1} and case (ii) in Lemma \ref{linearproximitygcd}.
 We also note that   
 $$\max_{1\leq i\leq n} h(u_i(\alpha))\leq h({\bf u}(\alpha))\leq n\cdot \max_{1\leq i\leq n} h(u_i(\alpha))\text{,}$$ 
 for moving points of the form 
 $\mathbf{u}=[1:u_1:\cdots:u_n]\text{,}$ 
 which are determined by maps
 $u_i \colon \Lambda \rightarrow \Osh_{\kk,S}^\times \text{,}$
 for $i = 1,\dots,n$.

\begin{proposition}\label{Moving:Laurent:Prop}
Let $\kk$ be a number field and $S$ a finite set of places of $\kk$, containing the archimedean places  and with ring of $S$-integers $\Osh_{\kk,S}$.   
Let  
$u_1,\hdots,u_n \colon A \rightarrow \Osh^\times_{\kk,S}$
be a sequence of maps with domain an infinite index set $A$.
If there exist 
$$(t_1,\hdots,t_n)\in\mathbb Z^n\setminus \{(0,\hdots,0)\}$$
such that 
$$
h\left( (u_1^{t_1} \cdot \hdots \cdot u_n^{t_n})(\alpha)\right)=\mathrm{o}\left(\max_{1\leq i\leq n} h(  u_i(\alpha))\right)  
$$
 for    each  $\alpha\in A$, then there exists   a proper algebraic subgroup $Z$ of $\mathbb G_m^n$ together with a map 
$\mathbf c: A \to \kk^\times$
with 
$
h(\mathbf c(\alpha))=\mathrm{o}\left(\max_{1\leq i\leq n} h(u_i(\alpha))\right)$
such that 
$
(u_1(\alpha),\hdots,u_n(\alpha)) \in \G^n_m(\Osh_{\kk,S})
$
is contained in $Z$ translated by $\mathbf c(\alpha)$ for each $\alpha\in A$. 
\end{proposition}
\begin{proof}
 We simply take $Z$ defined by  
$
 x_1^{t_1} \cdot \hdots \cdot x_n^{t_n}=1
 $
 and 
 $\mathbf{c}$  defined by 
$
u_1^{t_1} \cdot \hdots \cdot u_n^{t_n} \colon A \rightarrow \Osh_{\kk,S}^\times \text{.}
$
\end{proof}

\subsection{Proof of the key theorems}   To begin with, we establish Theorem \ref{Moving:GCDAffine:Claim:1}.
 
\begin{proof}[Proof of Theorem  \ref{Moving:GCDAffine:Claim:1}]
Let $\mathbf{u}=  (1,u_1,\hdots,u_n )$ and let $F$ and $G$ be the respective homogenizations of $f$ and $g$.  Then 
$$f_\alpha(u_1(\alpha),\hdots,u_n(\alpha))=F_\alpha(\mathbf{u}(\alpha)),$$ 
$$g_\alpha(u_1(\alpha),\hdots,u_n(\alpha))=G_\alpha(\mathbf{u}(\alpha))$$ 
and 
$$
\max_{1\leq i\leq n} \left( h(  u_i(\alpha)) \right) \leq h(\mathbf{u}(\alpha))\leq n\max_{1\leq i\leq n} \left( h(  u_i(\alpha)) \right)\text{, }
$$
for all $\alpha\in\Lambda$.
By replacing $F$ by $F^{\deg G}$ and $G$ by $G^{\deg F}$, we may assume that 
$$\deg F= \deg G=d\text{.}$$
We also use the Veronese embedding to view   the moving forms  $F_\alpha$ and $G_\alpha$ as hyperplanes in $\PP^{\binom{n+d}{d}-1 }({\kk})$.

Now, let $A\subseteq \Lambda$ be an infinite set which is coherent   with respect to $f$ and $g$.  Then $A$ is coherent with respect to the moving hyperplanes that are obtained from $F$ and $G$ via this Veronese embedding.
In particular, we may define a moving field with respect to $F$ and $G$ (as in     Definition \ref{moving:hypersurfaces:defn:3}).  Denote this field of moving functions by $\mathcal{K}_A$.  

Henceforth, we will identify  the restrictions of $F$ and $G$ to $A$ as polynomials in $\mathcal{K}_A[x_0,\dots,x_n]$.   
 Furthermore, these forms are coprime in $\mathcal{K}_A[x_0,\dots,x_n]$.  If not, then there exists a nonconstant homogeneous form   
 $Q \in\mathcal{K}_A[x_0,\dots,x_n]$ 
which is a common factor of both $F$ and $G$.
By the coherence property of $A$, the nonzero coefficients of $Q$ have finitely many zeros in $A$.  In particular, the moving polynomials 
$Q_{\alpha}\in \kk[x_0,\dots,x_n]$ 
are nonconstant for all but finitely many $\alpha\in A$.  Further, each such  moving polynomial $Q_{\alpha}$ is a common factor of  $F_{\alpha}$ and $G_{\alpha}$; we have obtained a contraction.    

Now, we fix a sufficiently large integer $m \gg 0$.  Let 
$$
V_m  := \mathcal{K}_A[x_0,\dots, x_n]_m / (F , G )_m,
$$
$$
N' = N_m' := \dim_{\mathcal{K}_A} V_m
$$
and put
$$N = N_m = \dim_{\mathcal{K}_A} (F , G )_m.$$   
Since $F$ and $G$ are coprime in  $\mathcal{K}_A[x_0,\dots, x_n]$,  a basic result in the theory of Hilbert functions gives 
$$
N' = \binom {m+n }{n } -2\binom {m+n-d }{n }+\binom {m+n-2d }{n } \text{, }
$$
see for example \cite[Proposition 12.11]{Hassett}.

Similarly,  as $F_{\alpha}$ and $G_{\alpha}$ are coprime in  $\kk[x_0,\dots, x_n]$ and of degree $d$, it follows that 
$$
 \dim_{\kk} V_m(\alpha)=N',
$$
where 
$$V_m(\alpha) = \kk[x_0,\dots,x_n]_m / (F_\alpha, G_\alpha)_m \text{, }$$ 
for each $\alpha \in A$.
Consequently, we have
$$
\dim_{\kk} (F_\alpha, G_\alpha )_m=N\text{, }
$$ 
for each $\alpha \in A$.

Next, given a monomial $\mathbf{x}^{\mathbf{i}}$, we use the same notation to denote its residue class modulo $(F , G )_m$.  We also denote by $\mathbf{u}^{\mathbf{i}}(\alpha)$ 
the evaluation of such monomials $\mathbf{x}^{\mathbf{i}}$ at the moving point $\mathbf{u}(\alpha)$.

For each $v \in S$ and each $\alpha \in A$, there exists a monomial basis $B_{v,\alpha}$ for $V_m$  which is then also a basis for $V_m(\alpha)$ that has the following two inductive properties 
\begin{itemize}
\item{
the monomial 
$$\mathbf{x}^{\mathbf{i}_1}  \in \kk[x_0,\dots,x_n]_m$$ 
is chosen so that $|\mathbf{u}^{\mathbf{i}_1}(\alpha)|_v$ is minimal subject to the condition that 
$$\mathbf{x}^{\mathbf{i}_1} \not \in (F_\alpha , G_\alpha )_m; \text{ and }$$
}
\item{
given monomials $\mathbf{x}^{\mathbf{i}_1},\dots,\mathbf{x}^{\mathbf{i}_j}$ that are linearly independent modulo $(F_\alpha, G_\alpha)_m$, choose a monomial $$\mathbf{x}^{\mathbf{i}_{j+1}} \in \kk[x_0,\dots,x_n]_m$$ with the property that $|\mathbf{u}^{\mathbf{i}_{j+1}}(\alpha)|_v$ is minimal subject to the condition that the monomials 
$$\mathbf{x}^{\mathbf{i}_1},\dots,\mathbf{x}^{\mathbf{i}_{j+1}}$$ 
are linearly independent modulo $(F_\alpha , G_\alpha )_m$.
}
\end{itemize}

Let  
$I_{v,\alpha} = \{\mathbf{i}_1,\dots,\mathbf{i}_{N'}\} $
be the set of exponent vectors for this monomial basis for $V_m$.  
 For each $\mathbf{i}$, with $|\mathbf{i}| = m$, there exists $c_{\mathbf{i},j} \in \mathcal{K}_A$ with the property that
$$
\mathbf{x}^{\mathbf{i}} + \sum_{j = 1}^{N'} c_{\mathbf{i},j} \mathbf{x}^{\mathbf{i}_j} \in (F , G )_m.
$$
Fix   a   $\mathcal{K}_A$-basis $\phi_1,\hdots,\phi_N$ for   the $\mathcal{K}_A$-vector space
$$(F , G )_m\subseteq\mathcal{K}_A[x_0,\dots,x_n]_m.$$  
In this way, we obtain, for each such $\mathbf{i}$, with $|\mathbf{i}| = m$, linear forms $L_{\mathbf{i},v,\alpha}$ over $\mathcal{K}_A$
\begin{equation}\label{moving:gcd:eqn3}
L_{\mathbf{i},v,\alpha}(\phi_1,\dots,\phi_N) = \mathbf{x}^{\mathbf{i}} + \sum_{j = 1}^{N'} c_{\mathbf{i},j} \mathbf{x}^{\mathbf{i}_j}.
\end{equation}

By evaluating the coefficients of the linear forms \eqref{moving:gcd:eqn3} at $\alpha \in A$, we obtain linearly independent linear forms
\begin{equation}\label{moving:gcd:eqn4}
 L_{\mathbf{i},v,\alpha}(\alpha)(\phi_1(\alpha),\dots,\phi_N(\alpha))= \mathbf{x}^{\mathbf{i}} + \sum_{j = 1}^{N'} c_{\mathbf{i},j}(\alpha)\mathbf{x}^{\mathbf{i}_j} \in (F_\alpha ,G_\alpha )_m,
\end{equation}
for each  $\alpha \in A$. (We replace $A$ by  a subset of $A$ with finite complement by the coherence property if necessary.) 

In particular, for each $\alpha\in A$, the set  
$\{L_{\mathbf{i},v,\alpha}(\alpha) : |\mathbf{i}| = m,\,\mathbf{i}\notin I_{v,\alpha}\}$
is a set of $\kk$-linearly independent forms in $N$ variables.  We note that there are only a finite number of choices for $I_{v,\alpha}$ as $v \in S$ and $\alpha \in A$ vary.  

Let $\mathcal H$ be the collection of (finitely many) hyperplanes defined by
\begin{align}\label{Hcal}
\mathcal H := \{L_{\mathbf{i},v,\alpha}  : |\mathbf{i}| = m,\,\mathbf{i}\notin I_{v,\alpha}\}
\end{align}
with $v$ running through $S$ and $\alpha$ running through $A$.  Since all of the coefficients of the linear forms defining 
$\mathcal H$ in \eqref{Hcal} are in $\mathcal K_A$, the field of moving functions $\mathcal K_{\mathcal H,A}$, with respect to $\mathcal H$ is a subfield of $\mathcal K_A$,  as in   Remark \ref{moving:hypersurface:remark}.

Let
$$
P(\alpha) = \phi(\mathbf{u}(\alpha)) := [\phi_1(\mathbf{u}(\alpha)):\dots:\phi_{N}(\mathbf{u}(\alpha))] \in \PP^{N-1}(\kk).
$$
We first consider the case where  the corresponding coordinate functions of  the moving points 
$$P=\phi:A\to \PP^{N-1}(\kk)$$
are  
degenerate with respect to the moving hyperplanes of $\mathcal H$.  In particular, the corresponding coordinate functions are linearly dependent over $\mathcal{K}_A$.  (Here, we replace $A$ by an infinite subset if necessary.)

To begin with, Lemma \ref{Mborel1} implies that there exist distinct   exponent vectors  
\begin{align}\label{ij}
\mathbf{i}_i=(i_0,\hdots,i_n) \text{ 
and }
\mathbf{i}_j=(j_0,\hdots,j_n) \text{,} \quad
\text{with }
 |\mathbf{i}_i|=|\mathbf{i}_j|=m, 
 \end{align}
such that
\begin{align}\label{hij}
h\left(\mathbf{u}^{\mathbf{i}_i}(\alpha)/\mathbf{u}^{\mathbf{i}_j}(\alpha)\right)=\mathrm{o}\left(\max_{1\leq i\leq n} h(  u_i(\alpha)) \right)
\end{align}
for  
 $\alpha$ in an infinite subset $A'$ of  $A$.
Indeed, this follows because    
\begin{align}\label{heightineq}
h\left(\left[ \mathbf{u}^{\mathbf{i}_1}(\alpha):\dots:\mathbf{u}^{\mathbf{i}_{N_m}}(\alpha)\right]\right)\leq m h(\mathbf{u}(\alpha))\leq mn\max_{1\leq i\leq n} h(  u_i(\alpha)).
\end{align}

Next, we  consider when  the corresponding coordinate functions of the moving points 
$$P=\phi:A\to \PP^{N-1}(\kk)$$
are  nondegenerate with respect to the moving hyperplanes of $\mathcal H$. 
Let $\epsilon > 0$.  We may apply  Theorem \ref{Moving:Schmidt:Subspace:Intro}  
to $\mathcal H$, the (finite) set of moving hyperplanes  to get 
\begin{equation}\label{moving:gcd:eqn21}
\sum_{v \in S} \sum_{\substack{
|\mathbf{i}| = m \\
\mathbf{i} \not \in I_{v,\alpha}
}
} \log \frac{ \|P(\alpha)\|_v }
{
|L_{\mathbf{i},\alpha,v}(\alpha)(P(\alpha))|_v 
} \leq (N+\epsilon) h(P(\alpha)) +\mathrm{o}(h(\mathbf{u}(\alpha)))
\end{equation}
for all $\alpha$ in an infinite subset  $A''$ of $A$.

Our main goal now, is to establish the following estimates
\begin{equation}\label{moving:gcd:eqn7}
  \sum_{v \in S} \sum_{\substack{
|\mathbf{i}| = m \\
\mathbf{i} \not \in I_{v,\alpha}
}
}
\log |L_{\mathbf{i},\alpha,v}(\alpha)(P(\alpha))|_v 
\leq   N' mn h(\mathbf{u}(\alpha)) +\mathrm{o}(h(\mathbf{u}(\alpha)))
\end{equation}
and
\begin{align}\label{moving:gcd:eqn8}
\begin{split}
N  h(P(\alpha)) - N\sum_{v \in M_\kk \setminus S}  \log^-  \max \{|F_{\alpha}(\mathbf{u}(\alpha))|_v, |G_{\alpha}(\mathbf{u}(\alpha))|_v  
\} & \\
\leq
\sum_{v \in S} \sum_{\substack{|\mathbf{i}| = m \\
\mathbf{i} \not \in I_{v,\alpha}} }
\log  ||P(\alpha)||_v 
+ \mathrm{o}(h(\mathbf{u}(\alpha))) 
\end{split}
\end{align}
 for all $\alpha \in  A''$.

 Together, these  estimates  \eqref{moving:gcd:eqn7} and \eqref{moving:gcd:eqn8}, yield the inequality
\begin{align}\label{moving:gcd:eqn9}
N  h(P(\alpha))& - N' mn h(\mathbf{u}(\alpha)) - N\sum_{v \in M_\kk \setminus S}  \log^- \max \{|F_{\alpha}(\mathbf{u}(\alpha))|_v, |G_{\alpha}(\mathbf{u}(\alpha))|_v \}  \cr 
&\leq 
\sum_{v \in S} \sum_{\substack{|\mathbf{i}| = m \\ 
\mathbf{i} \not \in I_{v,\alpha}}} \log \frac{ \|P(\alpha)\|_v }{ |L_{\mathbf{i},\alpha,v}(\alpha)(P(\alpha))|_v }+\mathrm{o}(h(\mathbf{u}(\alpha))).
\end{align}
Since  
$\phi_i \in   \mathcal{K}_A[x_0,\dots,x_n]_m$, 
we have that 
\begin{equation}\label{moving:gcd:eqn20}
h(P(\alpha)) \leq m h(\mathbf{u}(\alpha)) + \mathrm{o}(h(\mathbf{u}(\alpha))).
\end{equation}

Then by \eqref{moving:gcd:eqn21} and \eqref{moving:gcd:eqn9}, we have
\begin{align}\label{moving:gcd:eqn22}
\begin{split}
 - N\sum_{v \in M_\kk \setminus S} \log^-  \max\{
|F_{\alpha}(\mathbf{u}(\alpha))|_v,|G_{\alpha}(\mathbf{u}(\alpha))|_v 
\}  & \\  \leq (N'n+\epsilon ) m  h(\mathbf{u}(\alpha) ) +\mathrm{o}(h(\mathbf{u}(\alpha))).
\end{split}
\end{align}
By assumption, $F_\alpha(x)$ and $G_\alpha(x)$ are coprime.  The theory of Hilbert functions then implies that 
$$N' = \mathrm{O}(m^{n-2})$$ 
and 
$$N = \frac{m^n}{n!} + \mathrm{O}(m^{n-1}),$$ 
for $m \gg 0$.  Thus, if $\epsilon > 0$, then there exists $m \gg 0$, depending on $\epsilon$, so that \eqref{moving:gcd:eqn22} 
takes the form
$$
- \sum_{v \in M_\kk \setminus S} \log^- \max \{
|F_{\alpha}(\mathbf{u}(\alpha))|_v,|G_{\alpha}(\mathbf{u}(\alpha))|_v
\}  \leq \epsilon  h(\mathbf{u}(\alpha))  \text{, }
$$
for all $\alpha \in  A''$.

It is now left to show \eqref{moving:gcd:eqn7} and \eqref{moving:gcd:eqn8}.
 To this end, consider a place $v \in S$.    By construction of the monomials $\mathbf{x}^{\mathbf{i}_1},\dots,\mathbf{x}^{\mathbf{i}_{N'}}$  and \eqref{triangle}{\blue , } it follows that for all $\mathbf{i}$ with $|\mathbf{i}| = m$ and all $\mathbf{i} \not \in I_{v,\alpha}$ that

$$
\log |L_{\mathbf{i},\alpha,v}(\alpha) (P(\alpha))|_v \leq \log |\mathbf{u}(\alpha)^{\mathbf{i}} |_v + \log^+\max_{1\leq j\leq N'}|c_{\mathbf{i},j}(\alpha)|_v  + 2 \log (N'+1)   \text{.}
$$
Then
\begin{equation}\label{moving:gcd:eqn10}
- \sum_{v \in S} \sum_{\substack{
|\mathbf{i}| = m \\
\mathbf{i} \not \in I_{v,\alpha} 
}} \log|L_{\mathbf{i},\alpha,v}(\alpha)(P(\alpha))|_v 
\geq
- \sum_{v \in S} \sum_{\substack{
|\mathbf{i}| = m \\
\mathbf{i} \not \in I_{v,\alpha}
}
}
\log |\mathbf{u}^{\mathbf{i}}(\alpha)|_v - C(\alpha) N,
\end{equation}
where
$$
C(\alpha) = \sum_{v \in S}  \log^+\max_{1\leq j\leq N'}|c_{\mathbf{i},j}(\alpha)|_v +  2\# S =  \mathrm{o}(h(\mathbf{u}(\alpha))).
$$
Recall that $\mathbf{u}^{\mathbf{i}}(\alpha)$ is an $S$-unit.  The product formula then implies that
\begin{equation}\label{moving:gcd:eqn12}
\sum_{v \in S} \log |\mathbf{u}^{\mathbf{i}}(\alpha)|_v = \sum_{v \in M_\kk} \log |\mathbf{u}^{\mathbf{i}}(\alpha)|_v = 0 \text{.}
\end{equation}

Further
$$
- \sum_{v \in S} \sum_{
\substack{ 
|\mathbf{i}| = m \\
\mathbf{i} \not \in I_{v,\alpha}
} }
\log |\mathbf{u}^{\mathbf{i}}(\alpha)|_v 
=
- \sum_{v \in S} \sum_{|\mathbf{i} | = m }
\log | \mathbf{u}^{\mathbf{i}} (\alpha) |_v + \sum_{v \in S} \sum_{\mathbf{i} \in I_{v,\alpha} } \log |\mathbf{u}^{\mathbf{i}} (\alpha) |_v 
\text{, }
$$
which using \eqref{moving:gcd:eqn12},
simplifies to give
\begin{equation}\label{moving:gcd:eqn13}
- \sum_{v \in S} \sum_{
\substack{ 
|\mathbf{i}| = m \\
\mathbf{i} \not \in I_{v,\alpha}
} }
\log |\mathbf{u}^{\mathbf{i}}(\alpha)|_v 
= \sum_{v \in S} \sum_{\mathbf{i} \in I_{v,\alpha}} \log |\mathbf{u}^{\mathbf{i}}(\alpha)|_v.
\end{equation}
  Next we observe that  
\begin{align*}
 -\log |\mathbf{u}^{\mathbf{i}}(\alpha)|_v=\log \left| \frac 1{\mathbf{u}^{\mathbf{i}}(\alpha)} \right |_v\leq m\log \max \left\{ \left| \frac 1{u_0(\alpha)} \right|_v,\hdots, \left|\frac 1{u_n(\alpha)} \right|_v \right\},
\end{align*}
and hence
\begin{align}\label{uheight}
\begin{split}
\sum_{v\in S}  \sum_{\mathbf{i} \in I_{v,\alpha}}- \log \left|\mathbf{u}^{\mathbf{i}}(\alpha) \right|_v&\leq mN'\sum_{v\in S}\log \max \left\{ \left|\frac 1{u_0(\alpha)}\right|_v,\hdots,\left| \frac 1{u_n(\alpha)} \right|_v \right\}\cr
&= mN'h \left(\frac 1{u_0(\alpha)},\hdots,\frac 1{u_n(\alpha)}\right)\cr
&\leq mnN'h(\mathbf{u}(\alpha)).
\end{split}
\end{align}
Combining \eqref{moving:gcd:eqn10}, \eqref{moving:gcd:eqn13} and \eqref{uheight} we then obtain  
\begin{align*}
\sum_{v \in S} 
\sum_{\substack{|\mathbf{i}| = m \\
\mathbf{i} \not \in I_{v,\alpha}}} 
\log | L_{\mathbf{i},\alpha,v}(\alpha)(P(\alpha))|_v  
&
\leq \sum_{v \in S} \sum_{\substack{|\mathbf{i}| = m \\ \mathbf{i} \not \in I_{v,\alpha}}}
 \log |\mathbf{u}^{\mathbf{i}}(\alpha)|_v + C(\alpha) N 
 \\
 &
 = - \sum_{v \in S} \sum_{\mathbf{i} \in I_{v,\alpha}} \log |\mathbf{u}^{\mathbf{i}}(\alpha)|_v + C(\alpha) N 
 \\
  &
\leq N' mn h(\mathbf{u}(\alpha)) + \mathrm{o}(h(\mathbf{u}(\alpha))) \text{.}
\end{align*}
This establishes \eqref{moving:gcd:eqn7}.

Finally, we are going to show \eqref{moving:gcd:eqn8}.  First, we note
\begin{align}\label{moving:gcd:eqn16} 
\begin{split}
\sum_{v \in S} \sum_{\substack{
|\mathbf{i}| = m \\
\mathbf{i} \not \in I_{v,\alpha}
}} \log \|P(\alpha)\|_v \
& 
=
N \sum_{v \in S} \log  \|P(\alpha)\|_v 
\\ &
= N \left(
h(P(\alpha)) - \sum_{v \in M_\kk \setminus S} \log \|P(\alpha)\|_v
\right).
\end{split}
\end{align}
Now we observe that since
$$\phi_i \in (F,G)_m \subseteq \mathcal{K}_A[x_0,\dots,x_n]_m,$$
 we can write
$$
\phi_i(\alpha) = F_{\alpha} p_{i,\alpha} + G_{\alpha} q_{i,\alpha}
$$
for some 
$$p_{i,\alpha}, q_{i,\alpha}\in \kk[x_0,\dots,x_n].$$  
Thus, if $v \in M_\kk \setminus S$, then  
\begin{align*} 
 \log |\phi_i (\alpha)(\mathbf{u}(\alpha))|_v  \leq  \log  \max\{|F_{\alpha}(\mathbf{u}(\alpha))|_v, |G_{\alpha}(\mathbf{u}(\alpha))|_v \}  \\ +  \log  \max\{|p_{i,\alpha}(\mathbf{u}(\alpha))|_v, |q_{i,\alpha}(\mathbf{u}(\alpha))|_v \}.
 \end{align*}
By the identity 
$$\log (c) =\log^- (c) +  \log^+ (c) \text{, }$$ 
for each positive number $c$, and because of the fact that  
 $$|f(a_0,\hdots,a_n)|_v\leq { \|f\|_v} \text{,}$$ 
 if 
 $$f\in \kk[x_0,\hdots,x_n]\text{,}$$ 
 $v\notin S$  and each $a_i$ is an $S$-unit, the  
 above inequality becomes 
\begin{align}\label{moving:gcd:eqn18}
\begin{split}
 \log |\phi_i(\mathbf{u}(\alpha))|_v 
 &\leq \log^- \max\{|F_{\alpha}(\mathbf{u}(\alpha))|_v, |G_{\alpha}(\mathbf{u}(\alpha))|_v \}\cr
 & \quad+ \log^+ \max\{ \|F_{\alpha} \|_v, \|G_{\alpha} \|_v \}+ \log^+ \max\{\|p_{i,\alpha}\|_v, \|q_{i,\alpha}\|_v \}.
 \end{split}
\end{align}
Combining \eqref{moving:gcd:eqn16} and \eqref{moving:gcd:eqn18}, we then obtain that
\begin{align*}
\sum_{v \in S} \sum_{\substack{|\mathbf{i}| = m \\
\mathbf{i} \not \in I_{v,\alpha}} }
\log \|P(\alpha)\|_v 
\geq 
N \left(
h(P(\alpha)) - \sum_{v \in M_\kk \setminus S}  \log^- \max\{|F_{\alpha}(\mathbf{u}(\alpha))|_v, |G_{\alpha}(\mathbf{u}(\alpha))|_v \} 
 - C'(\alpha)
\right),
\end{align*}
for
\begin{align*}
C'(\alpha)  = \sum_{v \in M_\kk \setminus S} \left( \log^+ \max\{ \|F_{\alpha} \|_v, \|G_{\alpha} \|_v \}+ \log^+ \max\{\|p_{i,\alpha}\|_v, \|q_{i,\alpha}\|_v \}  \right)
=  \mathrm{o}(h(\mathbf{u}(\alpha))).
\end{align*}

This completes the proof of  \eqref{moving:gcd:eqn8}.   
In conclusion, we have shown  that in Theorem \ref{Moving:GCDAffine:Claim:1},  either the assertion (i)   or the following assertion (ii') below holds  for an infinite subset $A_1$ of  $\Lambda$  by \eqref{ij} and \eqref{hij}.

\bigskip
\noindent   (ii')  There exists an $(n+1)$-tuple of integers
\begin{align}\label{indexsets}
\mathbf{m} =(m_0,\hdots,m_n) \in\mathbb Z^{n+1}\setminus\{ (0,\hdots,0)\}
\end{align} 
 with 
$ \sum_{i=0}^n|m_i|\le 2m$,
and such that
\begin{align}\label{depndent}
h\left( (u_0^{m_0} \cdot \hdots \cdot u_n^{m_n})(\alpha) \right)=\mathrm{o}\left(\max_{1\leq i\leq n} h(  u_i(\alpha)) \right)\text{,} 
\end{align}
for   $\alpha\in A_1$.  We note for each $(n+1)$-tuple of integers $\mathbf{m} $
as in \eqref{indexsets}  we will always enlarge the index subset $A_1 \subseteq \Lambda$, if necessary, so that  it contains every $\alpha\in \Lambda$ that satisfies \eqref{depndent}.
 
We now  wish to strengthen this conclusion and show that  
there exist finitely many  infinite  subsets $A_1,\hdots,A_r$ of  $\Lambda$ such that  
$$\Lambda\setminus \cup_{j=1}^r A_i$$
is a finite set and (ii') holds for each $A_j$.  From now on,  we suppose  that the conclusion of (i) does not hold for a given $\epsilon>0$.
If 
$$\Lambda_1:=\Lambda\setminus A_1$$
is an infinite set, then our previous conclusion of (i) and (ii')  for  $\Lambda_1$  implies existence of an infinite index subset $A_2$ of $\Lambda_1$ such that (ii') holds for some   $(n+1)$-tuple of integers $\mathbf{m} $
as described in \eqref{indexsets}.  We also note that this pair  will be distinct from the one for $A_1$.  We can continue this process for  
$$\Lambda_2:=\Lambda\setminus (A_1\cup A_2)$$ 
and  then continue inductively.  Since there are only finitely many choices of index subsets  as in \eqref{indexsets}, this process will terminate in a finite number of steps until 
$$\Lambda\setminus   (A_1\cup A_2 \cup \hdots \cup A_{r})$$
is a finite set for some $r$. 

Finally, as the   height of  finitely many elements  is  bounded by a constant,   together with Proposition \ref{Moving:Laurent:Prop}, the above discussion implies existence of a  finite union of proper algebraic subgroups $Z$ of $\mathbb G_m^n$ together with  a map   
$$\mathbf c:  \Lambda \to \kk^\times \text{,}$$ 
with 
$$h(\mathbf c(\alpha))=\mathrm{o} \left(\max_{1\leq i\leq n} h(u_i(\alpha)) \right) \text{, }$$ 
 and such that 
 $(u_1(\alpha),\hdots,u_n(\alpha))$ is contained in $Z$ translated by   the  $\mathbf c(\alpha)$,  for  all $\alpha \in \Lambda$.  
\end{proof}

Next, we establish Theorem \ref{proximitygcd}.

 \begin{proof}[Proof of Theorem \ref{proximitygcd}]
 
By arguing as in the proof of Theorem \ref{Moving:GCDAffine:Claim:1}, it suffices to show either the assertion (i) holds or there exist  distinct exponent vectors 
\begin{align*}
\mathbf{i}(i)=(i_0,\hdots,i_n),\quad 
  \mathbf{i}(j)=(j_0,\hdots,j_n) \text{,} 
 \end{align*} 
 with 
$ |\mathbf{i}(i)|=|\mathbf{i}(j)|=m$,
and such that
\begin{align*} 
h\left(\mathbf{u}^{\mathbf{i}(i)}(\alpha)/\mathbf{u}^{\mathbf{i}(j)}(\alpha)\right)=\mathrm{o}\left(\max_{1\leq i\leq n} h(  u_i(\alpha)) \right)\text{,} 
\end{align*}
for    an infinite index subset $A \subseteq \Lambda$.

By assumption, $f_\alpha(x)$, for $\alpha \in \Lambda$, has nonzero constant term.  Let $d$ be the degree of $f_\alpha(x)$.  Then $d$ is independent of $\alpha \in \Lambda$, by assumption.  Note now that, by rearranging the index set in some order, if necessary, we may write  
$$
f_\alpha(u_1,\hdots,u_n)=a_{{\bf i}(0)}(\alpha)+\sum_{j=1}^{\ell}a_{{\bf i}(j)}(\alpha) {\bf u}^{{\bf i}(j)},
$$      
where, $\ell\le n$ and for each $0\leq j\leq \ell$, $a_{{\bf i}(j)}(\alpha)\neq 0$  for infinitely many $\alpha\in A$.

Replacing $\Lambda$ with an infinite subset if necessary, we may assume that 
$a_{{\bf i}(j)}(\alpha)\neq 0$  for all $\alpha\in \Lambda$ and each $0\leq j\leq \ell$.
We note that $\ell\geq 1$ since $\deg f_\alpha=d\geq 1$ for all $\alpha\in \Lambda$.
For later use, set
$$
\frak u:= \left(1,{\bf u}^{ {\bf i}(1)},\hdots,{\bf u}^{ {\bf i}(\ell)} \right).
$$ 
By evaluation at $\alpha \in \Lambda$, $\frak u$ determines a collection of moving points in $\PP^\ell$.  

Let $H_{\alpha} \subseteq \PP^{\ell}$, for $\alpha\in\Lambda$, be the hyperplane defined by 
$$L_\alpha:=\sum_{j=0}^{\ell}a_{{\bf i}(j)}(\alpha)X_j.$$ 
Then, by Lemma \ref{linearproximitygcd},   either
$$
\sum_{v \in S} \lambda_{H_{\alpha},S}(\frak u(\alpha))<   \epsilon h(\frak u(\alpha)) \leq d  \epsilon\max_{1\leq j\leq n} h(u_i(\alpha)) \text{, }
$$
for
$\alpha $ in an infinite subset $A$ of $\Lambda$;
or 
there exists $0\leq r\ne j\leq n$ such that 
 $$
 h( {\bf u}^{{\bf i}(j)}(\alpha)/{\bf u}^{{\bf i}(r)}(\alpha)) = \mathrm{o}\left(\max_{0\leq j\leq \ell} h({\bf u}^{{\bf i}(j)}(\alpha))\right)= \mathrm{o}\left(\max_{1\leq j\leq n} h(u_i(\alpha))\right) 
 $$
 for  $\alpha $ in an infinite subset $A'$ of $\Lambda$.

The second case  is our assertion at the beginning of the proof.
The first case implies our assertion (i) since  
$$L_\alpha(\frak u(\alpha)) =f_\alpha(u_1(\alpha),\hdots,u_n(\alpha))$$ 
and because of the fact that
\begin{align}\label{weilidentity}
- \log^- | L_\alpha(\frak u(\alpha))  |_v\leq \lambda_{H_{\alpha},v}(\frak u) +\mathrm{o}\left (\max_{1\leq i\leq n} h(u_i(\alpha)) \right).
\end{align}
Indeed, since
\begin{align}\label{weilinequality}
\lambda_{H_{\alpha},v}(\frak u (\alpha)) = \log \left( \frac{ \max_{0 \leq j \leq \ell} | {\bf u}^{ {\bf i}(j)}(\alpha) |_v   \max_{0 \leq j \leq \ell}  | a_{{\bf i}(j)}(\alpha) |_v}{ |  L_\alpha(\frak u(\alpha))|_v } \right) 
 \geq -2\log(\ell+1) \text{,}
\end{align}
\eqref{weilidentity} holds trivially if  
$\log |   L_\alpha(\frak u(\alpha)) |_v \geq 0 \text{,}$
since   
${\bf u}^{ {\bf i}(0)} =1 \text{,}$
whence 
$\max_{0 \leq j \leq \ell} | {\bf u}^{ {\bf i}(j)}(\alpha) |_v \geq 1\text{.}$
On the other hand, if 
$\log |   L_\alpha(u_1(\alpha),\hdots,u_n(\alpha)) |_v <0 \text{,}$ 
then, by \eqref{weilinequality},   we have that
$$
-\log^- | L_\alpha(u_1(\alpha),\hdots,u_n(\alpha)) |_v-\lambda_{H_{\alpha},v}(\frak u) \leq   -\log \max_{0\leq j \leq \ell}  | a_{{\bf i}(j)}(\alpha) |_v +2\log(\ell+1).
$$
Finally, since 
$a_{{\bf i}(j)}(\alpha)\ne 0\text{, }$ 
for all   $\alpha\in \Lambda $, we obtain that
\begin{align}
\sum_{v\in S}-\log|a_{{\bf i}(j)}(\alpha)|_v 
\leq \sum_{v\in S}-\log^-| a_{{\bf i}(j)} (\alpha)|_v 
 \leq h(a_{{\bf i}(j)}(\alpha))
=
\mathrm{o}\left (\max_{1\leq i\leq n} h(u_i(\alpha)) \right)\text{.}
\end{align}
This completes the proof.
 \end{proof}
  
\section{
The GCD problem for pairs of linear recurrence sequences
}\label{linear:recurrence}

In this section, we prove Theorems \ref{gcdrecurrencebasic}, \ref{gcdrecurrence:intro} and Proposition \ref{OS}.
We first prove the following lemma, which is the moving target analogue of \cite[Lemma 5.2]{Levin:GCD}.

\begin{lemma}\label{proximity}
Let  
$G(n) = \sum_{i = 1}^r q_i(n)\beta_i^n$
be a nondegenerate algebraic linear recurrence
sequence defined over a number field $\kk$.  
Let $v\in M_{\kk}$ be such that $|\beta_i|_v\geq 1$ for some $i$.  
Let $\epsilon > 0$.  Then 
\begin{align}\label{assertion1}
-\log^- |G(n)|_v<\epsilon n
\end{align}
for all but finitely many $n\in\mathbb N$.
 \end{lemma}
  \begin{proof} 
  
 It suffices to show that for any infinite subset $\Lambda$ of $\mathbb N$,  there are  infinitely  many $n\in\Lambda$ such that \eqref{assertion1} holds.  
Clearly, it leads to a contradiction if  the number of $n\in\mathbb N$ such that  \eqref{assertion1} fails is infinite.  We first note if $r=1$, that is if 
 $$G(n) = q_1(n)\beta_1^n \text{,}$$ 
 then the assumption on $v$ reads  $|\beta_1|_v\geq 1$, and hence  $$-\log^- |G(n)|_v=-\log^- |q_1(n))|_v\le  \epsilon n$$ 
for all $n$  sufficiently large.  Therefore, we may assume that $r\geq 2.$
 
Let $H_n \subseteq \PP^{r-1}$ be the moving hyperplane  defined by
$$
q_1(n)x_1 + \hdots + q_r(n) x_r = 0,
$$
 for  $n\in\mathbb N$.
Furthermore, consider the moving points
$$\beta(n) = [\beta_1^n: \dots : \beta_r^n]:\mathbb N\to \PP^{r-1}(\kk).$$
 By assumption, $G(n)$ is a nondegenerate linear recurrence sequence.  
 Thus, $\beta_i/\beta_1$ is not a root of unity for $i\geq 2$.  It also follows that  $h(\beta_1,\hdots,\beta_r)>1$; whence
$h(H_n) = \mathrm{o}(h(\beta(n)))$.
Let $\epsilon>0$.
Apply Lemma \ref {linearproximitygcd} for the
 case of the infinite subset $\Lambda$ of $\mathbb N$. We obtain that either
\begin{equation}\label{logGn:0}
\lambda_{H_{n},v}(\beta(n))< \epsilon nh(\beta_1,\hdots,\beta_r)
\end{equation}
for  infinitely  many $n\in\Lambda$, or
there exists $0\leq i\ne j\leq r$ such that 
 \begin{equation}\label{logGn:1}
  h(\beta_i^n /\beta_j^n) = \mathrm{o}(h(\beta_1^n,\hdots,\beta_r^n)) 
 \end{equation}
 for 
infinitely  many $n\in\Lambda$.
 In fact, the second possibility \eqref{logGn:1} cannot occur since $\beta_i/\beta_j$ is not a root of unity.

 Thus, because of \eqref{logGn:0}, it remains to establish the inequality 
 \begin{equation}\label{logGn:2}
 - \log^- | G(n) |_v\leq \lambda_{H_{n},v}(\beta(n)) + \mathrm{O}(\log n).
 \end{equation}
To this end, since
\begin{align*}
 \lambda_{H_{n},v}(\beta(n)) & = \log \frac{ \max_i | \beta_i^n |_v  \max_i | 
q_i(n) |_v}{ | q_1(n) \beta_1^n + \hdots + q_r(n) \beta_r^n|_v }
\\
& = \log \frac{ \max_i | \beta_i^n |_v  \max_i | q_i(n) |_v}{ | G(n)|_v } \\
& \geq  -2\log r,
\end{align*}
the inequality \eqref{logGn:2} holds trivially if  
$\log | G(n) |_v \geq 0 \text{.}$  

On the other hand, since  
$\max_i | \beta_i  |_v \geq 1\text{,}$ by assumption, when 
$\log | G(n) |_v<0\text{,}$ 
 we have that 
\begin{align}\label{logGn}
\begin{split}
-\log^- | G(n) |_v-\lambda_{H_{t+1},v}(\beta(n)) & = -\log  \max_i | \beta_i^n |_v -\log \max_i | q_i(n) |_v  \\ 
&
\leq  -\log \max_i | q_i(n) |_v.
\end{split}
\end{align}
Finally, observe that for all $n$ such that $q_i(n)\ne 0$ 
$$-\log| q_i(n) |_v\leq -\log^-| q_i(n) |_v\leq h(q_i(n))=  \mathrm{O} (\log n);$$
there are at most finitely many $n$ such that $q_i(n)=0$.
The desired inequality \eqref{logGn:2} is now a consequence of the inequality \eqref{logGn}.
 \end{proof}

In our proof of Theorem \ref{gcdrecurrencecounting}, we make use of Proposition \ref{moving:Prop} below.  

\begin{proposition}\label{moving:Prop}
Let $f_1, f_2 \in  \kk[t,x_1,\dots,x_r]$ be coprime  polynomials and assume that $f_1$ has positive degree in at least one of the variables $x_i$ and that $f_2$ has positive degree in at least one of the variables $x_j$.  
Then, the polynomials  
$f_1(n), f_2(n)\in \kk [x_1,\dots,x_r]$ 
are coprime for 
  all but perhaps finitely many $n\in \mathbb N$.
\end{proposition}

\begin{proof}
 
Let 
$F_1 \text{ and } F_2\in  \kk[t][x_0,\dots,x_r]$
be the respective homogenization of $f_1$ and  $f_2$ with respect to the variable $x_0$.  By assumption, $f_1$ and $f_2$ are coprime and so the same is true for their homogenizations with respect to $x_0$.  In particular, $F_1$ and  $F_2 $ are coprime in $\kk[t][x_0,\dots,x_r]$ and so  their common zero set has codimension 2 in $\mathbb P^r( \overline{\kk(t)})$.

Therefore, we may find linear forms
$ L_1,\hdots,L_{r-1 }\in\kk[x_{0 },\dots,x_{r}] \text{,}$
which have the property that 
 $$F_1,F_2, L_1,\hdots,L_{r-1}\in\kk[t] [x_0,\dots,x_r] \subseteq \kk(t)[x_0,\dots,x_r]$$ 
have no common zero in $\mathbb P^r( \overline{\kk(t)})$.   

By the theory of resultants, for example \cite[Chapter IX]{Lang}, the resultant 
$$R(F_1,F_2,L_1,\hdots,L_{r-1})\in \kk[t]$$ 
is not zero, and, hence, it has only finitely many zeros in $\kk$. 
By evaluating this polynomial at $n \in \NN$, it follows that    
$$R(F_1(n),F_2(n),L_1,\hdots,L_{r-1})\ne 0$$ 
for all but finitely many $n\in\mathbb N$.   

On the other hand, if 
$f_1(n) \text{ and } f_2(n)\in\kk[x_1,\dots,x_r]$ 
have a nonconstant common factor, then  the forms 
$F_1(n) \text{ and } F_2(n)\in\kk[x_0,\dots,x_r]$
 have a nonconstant homogeneous  common factor 
 $H(n) \in  \kk[x_{0},\dots,x_r].$

Now, given such a nonconstant common factor $H(n) $, note that, for dimension reasons, $H(n),L_1,\hdots,L_{r-1}$ must have a common zero in $\mathbb P^r( \overline{\kk})$.  Since $H(n)$ is a nonconstant common factor of $F_1(n)$ and $F_2(n)$, such a common zero is also a common zero of $F_1(n),F_2(n),L_1,\hdots, \text{ and } L_{r-1}$  in $\mathbb P^r( \overline{\kk})$.  Consequently, $$R(F_1(n),F_2(n),L_1,\hdots,L_{r-1})=0 \text{, }$$
for all such $n \in \NN$.  

In conclusion, it follows that the polynomials
 $f_1(n) \text{ and } f_2(n)\in\kk[x_1,\dots,x_r]$ are coprime 
 for all but finitely 
 many $n\in\mathbb N$.   
 \end{proof}
 
The following theorem is analogous to \cite[Theorem 5.3]{Levin:GCD}.    Here, we use it to establish  Theorems \ref{gcdrecurrencebasic} and \ref{gcdrecurrence:intro} in addition to Proposition \ref{OS}.  (See Proposition \ref{OS:2}.).  
\begin{theorem}\label{gcdrecurrencecounting}
Let $\kk$ be a number field and $S$ be a finite set of places of $\kk$, containing the archimedean places, and let $\Osh_{\kk.S}$ be the ring of $S$-integers.
Let  
$F(m) = \sum_{i = 1}^s p_i(m)\alpha_i^{m}$
and 
$G(n) = \sum_{i = 1}^t q_i(n)\beta_i^n$
be algebraic  linear recurrence sequences, defined  over $\kk$, and such that their roots are in $\Osh_{\kk,S}^\times$ and generate together a torsion-free multiplicative group $\Gamma$. 
  Let $\epsilon>0$ and consider the inequality
\begin{align}\label{assertion2}
 \sum_{v\in M_{\kk}\setminus S}-\log^-\max\{|F(m)|_v, |G(n)|_v\}>\epsilon  \max\{n,m\} 
 \end{align} 
 for pairs of positive integers   $(m,n)\in \NN^2$.  The following assertions hold true.
 \begin{enumerate}
 \item Consider the case that $m=n$.  If  the inequality \eqref{assertion2} is valid for infinitely many positive integers $(n,n)\in\mathbb \NN^2$, then the linear recurrences $F$ and $G$ have a non-trivial common factor  in the ring of linear recurrences $R_{\Gamma}$.
 \item  Consider the case that $m \not = n$. If  the inequality \eqref{assertion2} is valid for infinitely many  pairs of positive integers  $(m,n)\in \NN^2$,  with $m\ne n$, then  
  $F$ and $G$  are not separated (see Definition \ref{rootset}) and  there exist  finitely many  pairs  of nonzero  integers $(a_i,b_i)\in \ZZ^2$, for $i=1,\hdots,c$, such that  for  $m$ or $n$ sufficiently large   the pair  $(m,n)$ satisfies one of the following relations
 $$|ma_i+nb_i| =\mathrm{o} (\max\{m,n\})\text{,}$$
 for $1\leq i\leq  c$.
 \end{enumerate}
\end{theorem}

\begin{proof}
Let  $\Gamma$ be the torsion free group of rank $r$ generated by the combined roots of the recurrence sequences $F(n)$ and $G(n)$.  Let $u_1,\dots,u_r$  be    multiplicatively independent generators for $\Gamma$ and let
$$f,g \in \kk \left[t, x_1^{\pm 1},\hdots,x_r^{\pm 1} \right]$$ 
be the Laurent polynomials corresponding to $F$ and $G$.  
We may write 
$$
f(t,x_1,\hdots,x_r) = x_1^{i_1} \cdot \hdots \cdot x_r^{i_r} f_0(t,x_1,\hdots,x_r)
$$
and
$$
g(t,x_1,\hdots,x_r) = x_1^{j_1} \cdot \hdots \cdot x_r^{j_r} g_0(t,x_1,\hdots,x_r) \text{,}
$$
where $i_1,\hdots,i_r,j_1,\hdots,j_r\in\mathbb Z$ and where 
$f_0,g_0\in \kk[t, x_1,\hdots,x_r] = \kk[t][x_1,\hdots,x_r] \text{, }$
with 
$x_i\nmid f_0g_0,$ 
for $1\leq i\leq r$.  

Let $F_0(n)$ and $G_0(n)$ be the linear recurrences that are determined by $f_0$ and $g_0$.    Then  we may write  
$$\alpha_i=\prod_{j=1}^r u_j^{i_{j}} $$ and similarly for the $\beta_j$.  
Under this convention, 
we see that 
\begin{align}\label{F0G0}
F(n)=u_1^{ni_1} \cdot \hdots \cdot u_r^{ni_r} F_0(n) \qquad \text{and}\quad G(n)=u_1^{nj_1} \cdot \hdots \cdot u_r^{nj_r} G_0(n).
\end{align}
Furthermore, since  $u_1,\hdots,u_r \in \mathcal O^\times_{\kk,S}$, it follows, from \eqref{F0G0}, that
\begin{align}
 \sum_{v\in M_{\kk} \setminus S}-\log^-\max\{|F(m)|_v, |G(n)|_v\} 
  = \sum_{v\in M_{\kk} \setminus S}-\log^-\max\{|F_0(m)|_v, |G_0(n)|_v\}.
 \end{align}
 Therefore,  in our study of the inequality \eqref{assertion2},  without loss of generality, we may assume that 
 $f,g\in \kk[t][x_1,\hdots,x_r]=\kk[t,x_1,\hdots,x_r]$ 
 and 
 $x_i\nmid fg$ 
 by replacing $F$ and $G$ by $F_0$ and $G_0$.
 
We will first consider   the case that  $n=m$.   By assumption, there exists an infinite index subset $\Lambda$ of $\mathbb N$ such that   the inequality 
 \begin{align}\label{assertion3}
 \sum_{v\in M_{\kk}\setminus S}-\log^-\max\{|F({n})|_v, |G(n)|_v\}>\epsilon  n
 \end{align} 
   is valid for all  $n\in\Lambda$.  
 Assume that $F$ and $G$ are coprime in $R_{\Gamma}$.
It  follows that $f$ and $g$ are coprime polynomials in $\kk[t, x_1,\hdots,x_r]$.  Then, by Proposition \ref{moving:Prop}, the polynomials
$$f(n,x_1,\dots,x_r) \text{ and } g(n,x_0,\dots,x_r) \in \kk[x_1,\dots,x_r]$$ 
are coprime for all but finitely many $n \in  \mathbb N$. 
 
Denote by 
$\mathbf{u}(n):=(u_1^n,\dots,u_{r}^n)$, 
for $n\in\NN$.  Then $\mathbf{u}$ can be viewed as a map from $\NN$ to $\kk$. 
 We  apply Theorem \ref{Moving:GCDAffine:Claim:1} 
  to the moving forms 
$$f(n,x_1,\dots,x_r) \text{ and } g(n,x_1,\dots,x_r) \in \kk[x_1,\dots,x_r],$$
for $n \in \Lambda$, by setting  
$$\epsilon_0=  \epsilon /\max\{h(u_1),\dots,h(u_r)\} >0\text{.}$$
 
 Then, since $u_1^{i_1}\cdot \hdots \cdot u_r^{i_r}$ is not a root of unity for all 
  $(i_1,\hdots,i_r)\ne(0,\hdots,0)\in \mathbb Z^r$,   by  arguing 
similar to the proof of Lemma \ref{proximity},   we   deduce  that the moving polynomials $f$ and $g$ have slow growth with respect to $\mathbf{u}(n)$ for $n$ sufficiently large.   Furthermore,  the conclusion of Theorem \ref{Moving:GCDAffine:Claim:1} (ii) does not hold. 
  Suppose, on the other hand, that the conclusion of Theorem  \ref{Moving:GCDAffine:Claim:1}  (i) does hold true in our present context.  Then
\begin{align}
 \sum_{v\in M_{\kk} \setminus S}-\log^-\max\{|f(n,u_1^n,\dots,u_r^n)|_v, |g(n,u_1^n,\dots,u_r^n)|_v\} 
   <  
\epsilon_0 n\cdot \max\{h(u_1),\dots,h(u_r)\} = \epsilon n   
 \end{align}
 for infinitely many $n\in\Lambda$,  which  clearly contradicts  \eqref{assertion3} as 
 $$F(n)=f(n,u_1^n,\dots,u_r^n)$$
and
$$G(n)=g(n,u_1^n,\dots,u_r^n)\text{.}$$ 
This shows that $F$ and $G$ cannot be coprime in $R_{\Gamma}$  and completes the proof of (i).

 We now treat the case when $m\ne n$. This is done by adapting the method for simple recurrence sequences \cite[Theorem 5.3]{Levin:GCD}.  Define polynomials 
 $$
 \tilde f(t_1,t_2,x_1,\hdots,x_{2r}) \text{, }
  \tilde g( t_1,t_2,x_1,\hdots,x_{2r}) \in \kk[t_1,t_2,x_1,\dots,x_{2r}]
 $$
 by the condition that
\begin{align*}
 \tilde f( t_1,t_2,x_1,\hdots,x_{2r})&=f(   t_1,x_1,\hdots,x_{r})\cr
  \tilde g( t_1,t_2,x_1,\hdots,x_{2r})&= g(  t_2,x_{r+1},\hdots,x_{2r}).
\end{align*}
Then $\tilde f( m,n,x_1,\hdots,x_{2r})$ and $\tilde g(m,n,x_1,\hdots,x_{2r})$ are coprime in $\kk[x_1,\hdots,x_{2r}]$, for  all but finitely many $m,n\in\mathbb N$, since they have  no variable in common.

Let
\begin{align*}
\tilde{\mathbf{u}}(m,n) =(u_1^m,\hdots,u_r^m,u_1^n,\hdots,u_r^n) \text{,}
\end{align*}
 for $ m,n\in\mathbb N$. 
Then $\tilde{\mathbf{u}}$  can be viewed as a map from the (double) index set $\NN^2$ to $\kk$. 
 By the assumption of (ii), there exists an infinite subset  
 $$\Lambda_0\subset \NN^2 \setminus\{(n,n) : n \in \NN\}$$
 such that the inequality \eqref{assertion2} holds for all $(m,n)\in\Lambda_0$.
Again, we apply Theorem \ref{Moving:GCDAffine:Claim:1}  to the moving polynomials 
$$\tilde f(m, n, x_1,\hdots,x_{2r}),\, \tilde g(m, n, x_1,\hdots,x_{2r}) \in \kk[x_1,\hdots,x_{2r}]\text{,}$$ 
which  we may assume  are coprime for all  $(m,n)\in \Lambda_0$. 

  Since $\Lambda_0$ is chosen so that  the inequality \eqref{assertion2} holds for all $(m,n)\in\Lambda_0$,
we see that the conclusion of Theorem \ref{Moving:GCDAffine:Claim:1} (i) does not hold.   Therefore, by the conclusion of Theorem \ref{Moving:GCDAffine:Claim:1} (ii),
 there exists   a finite union of proper algebraic subgroups  
 $$Z \subsetneq \mathbb G_m^{2r}$$ 
 together with  a map  
$$\mathbf c: \Lambda_0 \to \kk^\times \text{,}$$ 
 which, for all $(m,n) \in \Lambda_0$,   have  the two properties that:
\begin{itemize}
\item{ 
$h(\mathbf c(m,n))=\mathrm{o} \left(\max\{m,n\} \right)$; and}
\item{
 $\tilde{\mathbf{u}}(m,n)$ is contained in $Z$ translated by $\mathbf c(m,n)$.}
 \end{itemize}
  
Now, since $u_1,\hdots,u_r$ are multiplicatively independent, such a    $Z$ must be a finite union of  proper algebraic subgroups   of $\mathbb G_m^{2r}$ which are of the form 
 
\begin{equation}\label{exceptional:subgroup}
x_i^ax_{i+r}^b=1\text{,}
\end{equation} 
for $i = 1,\dots,r$,
where 
$$0 \not = a,b\in\mathbb Z$$ 
have the property that there exists infinitely many $(m,n)\in \Lambda_0$ such that 
$$u_i^{ma+nb}=\mathbf c(m,n)\text{.}$$ 
Hence
$$h(u_i^{ma+nb})=\mathrm{o} (\max\{m,n\}).$$
Further,  since  
$$h(u_i^{ma+nb})=|ma+nb|\cdot  h(u_i)\text{,}$$ 
it follows that  
$$|ma+nb| =\mathrm{o} (\max\{m,n\}) \text{.}$$

Finally,  suppose that $F$ and $G$ 
are separated.  Then each $u_i$ can be in only one of $\Gamma_F$ and $\Gamma_G$, the groups generated by the respective roots of $F$ and $G$. 
 But then this means that the relations \eqref{exceptional:subgroup} cannot occur.  This contraction establishes (ii) and concludes the proof.
\end{proof}

We obtain Theorem \ref{gcdrecurrencebasic} by combining Lemma \ref{proximity} and Theorem \ref{gcdrecurrencecounting}.

\begin{proof}[Proof of Theorem \ref{gcdrecurrencebasic}]
   Let $S$ be a finite set of places of $\kk$, containing the archimedean places, and such that $\alpha_1, \dots, \alpha_s, \beta_1,\dots, \beta_t$, the respective roots of $F$ and $G$, are in $\Osh_{\kk,S}^\times$.  Let $\epsilon>0$.  
  
Fix $v\in S $.  Then, by assumption, $\max\{ |\alpha_i|_v\}\geq 1$ or $\max\{ |\beta_j|_v\}\geq 1$.
Moreover, as $S$ is a finite set, we can successively apply Lemma \ref{proximity} for each $v\in S$.  In doing so, we obtain validity of the inequality
\begin{align}\label{B}
 \sum_{v\in  S} -\log^-\max\{|F(m)|_v, |G(n)|_v\}<\frac{\epsilon}2 \max\{m,n\}
 \end{align} 
for all but finitely  many pairs $(m,n)\in\mathbb N^2$. 

In particular, if the inequality
$$
 \log  \gcd(  F(  m ),  G(  n )  ) >\epsilon \max\{  m,n \}
$$
has infinitely many solutions $(m,n)\in \NN^2$, then the inequality \eqref{B} implies that
\begin{align}\label{A}
 \sum_{v\in M_{\kk}\setminus S}-\log^-\max\{|F(m)|_v, |G(n)|_v\}>\frac{\epsilon}2 \max\{m,n\},
 \end{align} 
 for infinitely many $(m,n) \in \NN^2$.   Both of the conclusions (i) and (ii) of Theorem \ref{gcdrecurrencebasic} are now evident consequences of Theorem \ref{gcdrecurrencecounting} applied to the case that $\epsilon_0 = \epsilon / 2 > 0$.
\end{proof}

Theorem \ref{gcdrecurrence:intro} is a consequence of Theorem \ref{gcdrecurrencebasic}.

\begin{proof}[Proof of Theorem \ref{gcdrecurrence:intro}]
  To establish (i), first note that if the group generated by the roots of $F$ and $G$ has a torsion subgroup, say of order $q$, then the recurrences 
  \begin{align}\label{qrecurence}
 F_\ell(n):=F(qn+ \ell) \quad
{\text and }\quad 
 G_\ell(n):=G(qn+ \ell) 
\end{align}
  have roots generating a torsion-free group $\Gamma_\ell$, for $0\leq \ell \leq q-1$. 
  Therefore, we may assume that $\Gamma$ is torsion free. 
 
  Let $\epsilon>0$.  It then follows from Theorem \ref{gcdrecurrencebasic} (i) that if 
 the inequality
$$
 \log  \gcd(  F_\ell(n),  G_\ell(n)  ) >\epsilon n
$$
has infinitely many solutions $n \in \NN$, then all but finitely many of them satisfy one of finitely many linear relations
 $$
 (m,n)=(a_it+b_i,c_it+d_i),\quad t\in\mathbb Z, \, i=1,\hdots,\ell,
 $$
 where $a_i, b_i, c_i, d_i\in\mathbb Z$, $a_ic_i\ne 0$.  Furthermore, the linear recurrences $F(a_i \bullet +b_i)$ and $G(c_i\bullet +d_i)$ have a non-trivial common factor for $i=1,\hdots,\ell$.
 
Finally,  if   $F$ and $G$  are separated,  in the sense of Definition \ref{rootset}, then $\Gamma_F$ and  $\Gamma_G$ have trivial intersection. Further, it follows that  the  linear recurrences $F_\ell$ and $G_\ell$, $1\leq \ell\leq q-1$, in \eqref{qrecurence}, are   separated and  are coprime.  Therefore, Theorem \ref{gcdrecurrence:intro} (ii) is  implied by Theorem \ref{gcdrecurrence:intro} (i) for the case that $m=n$ combined with Theorem \ref{gcdrecurrencebasic} (ii) for the case that $m\ne n$.
  \end{proof}   

We will prove the following proposition, which implies Proposition \ref{OS}.

\begin{proposition}\label{OS:2} 
Let $\kk$ be a number field and $S$ a finite set of places of $\kk$, containing the archimedean places and having ring of $S$-integers $\Osh_{\kk,S}$.  Let 
$F(m) $  and 
$G(n)$ be   linear recurrence sequences with roots and coefficients  in $\kk$.    Suppose  that the roots of $F$ and $G$ generate together a torsion-free multiplicative subgroup $\Gamma$ of $\kk^\times$.  Suppose furthermore that  
$G$ has more than one root.  Then the following assertions hold true.
 \begin{enumerate}
 \item{  Consider the case that $m = n$.  Suppose that $F$ and $G$ are coprime (with respect to $\Gamma$).  Then 
there exist at most finitely many  natural numbers  $n\in \NN$,    which have the properties that  
$F(n)/G(n)\in\mathcal O_{\kk,S}$.}
\item{  Consider the case that $m \not = n$. 
If there are infinitely many pairs  of natural numbers  $(m,n)\in \NN^2$ with $m<n$   and  
$F(m)/G(n)\in\mathcal O_{\kk,S}$, then  
 the linear recurrences $F$ and $G$  are not separated.  Further, in this case,  there exists 
 finitely many
   pairs   of   nonzero   integers $(a,b)\in \ZZ^2$  such that  
$$|ma+nb| =\mathrm{o} (n) \text{,}$$
as  $m$  becomes  sufficiently large. 
In particular, there   does not exist  infinitely many pairs of natural numbers $(m,n)\in \NN^2$ with $m=\mathrm{o} (n)$  and  
$F(m)/G(n)\in\mathcal O_{\kk,S}$.
}
\end{enumerate}
\end{proposition}

  Propositions \ref{OS} and \ref{OS:2} are consequences of Theorem \ref{gcdrecurrencecounting} and Lemma \ref{linearproximitygcd}.

\begin{proof}[Proof of Propositions \ref{OS} and \ref{OS:2}] 
Let  
$F(n) = \sum_{i = 1}^s p_i(n)\alpha_i^n$
and 
$G(n) = \sum_{i = 1}^t q_i(n)\beta_i^n\text{,}$ 
for $n \in \mathbb{N}$.
Without loss of generality we may enlarge $S$ and assume that it is  a finite set of places of $\kk$, containing the archimedean places such that all the roots and nonzero coefficients of $F$ and $G$ are in $\mathcal O^\times_{\kk,S}$.  Moreover, we can also assume that $\alpha_1=\beta_1=1$ by dividing $F(n)$ by $\alpha_1^n$ and $G(n)$ by $\beta_1^n$ without changing the following set 
$$
\Lambda:= \left \{(m,n)\in\mathbb N^2\, :  m  \leq  n  \text{ and } \frac{F(m)}{G(n)}\in \mathcal O_{\kk,S}\right \}.
$$
Let $\epsilon >0$.
We first consider the case that there are infinitely many $(n,n)\in \Lambda$.
Since $F$ and $G$ are coprime, Theorem \ref{gcdrecurrencecounting} (i) implies that  
\begin{align}\label{recurrence:eqn:proof1}
\sum_{v\in M_{\kk}\setminus S}-\log |G(n)|_v 
= 
\sum_{v\in M_{\kk}\setminus S}-\log^- |G(n)|_v 
= 
\sum_{v\in M_{\kk}\setminus S}-\log^-\max\{|F(n)|_v, |G(n)|_v\}  
< \epsilon   n
\end{align}
for all but finitely many $(n,n)\in\Lambda$.
Next, we consider when
there are infinitely many pairs $(m,n)\in\Lambda$ with $m< n$ and suppose the conclusion of (ii) does not hold.  Then by Theorem \ref{gcdrecurrencecounting} (ii), we have similarly
\begin{align}\label{recurrence:eqn:proof1.2}
\sum_{v\in M_{\kk}\setminus S}-\log |G(n)|_v =-\log^-\max\{|F(m)|_v, |G(n)|_v\}  
< \epsilon \max\{m,n\}=\epsilon n. 
\end{align}
for all but finitely many $(m,n)\in\Lambda$.      We now consider, simultaneously, consequences of the two inequalities \eqref{recurrence:eqn:proof1} and \eqref{recurrence:eqn:proof1.2}. 

Let $H_n \subseteq \PP^{t-1}$ be the moving hyperplane defined by
$
q_1(n)x_1 + \hdots + q_t(n) x_t = 0.
$
Furthermore, consider the moving points
$$
\beta(n) = [\beta_1^n: \dots : \beta_t^n]:\mathbb N\to \PP^{t-1}(\kk) \text{,}
$$ 
where $\beta_1=1$.
We note that $t\ge 2$ since $G$ has more than one root.

For $v\in M_{\kk}\setminus S$,
\begin{align}\label{notS}
\lambda_{H_{n},v}(\beta(n)) & := \log \left(\frac{\|\beta(n)  \|_v\cdot \|H_{n} \|_v}{|q_1(n)\beta_1^n + \hdots + q_t(n) \beta_t^n |_v} \right) 
= \log    \|H_{n}\|_v -\log |G(n) |_v   
<\epsilon n 
\end{align}
for all but finitely many $n$ that satisfies  \eqref{recurrence:eqn:proof1} or  \eqref{recurrence:eqn:proof1.2}.

 Since $\Gamma$ is torsion free and $G(n)$ is a linear recurrence sequence having more than one root, $\beta_i/\beta_j$ is not a root of unity for $i\ne j$.  
Therefore, the growth of $h(\beta_i^n/\beta_j^n)$ is the same as $h(\beta_1^n,\hdots,\beta_t^n)$.  We wish to apply Lemma \ref{linearproximitygcd}.   First, we  note that case (ii)   of Lemma \ref{linearproximitygcd} cannot occur similar to the proof of Lemma \ref {proximity}.  Then, by Lemma \ref{linearproximitygcd}, for  
 $\epsilon_0=\epsilon/|S|\text{,}$
\begin{align}\label{inS}
\lambda_{H_{n},v}(\beta(n))< \epsilon_0 n 
\end{align}
for infinitely many $n\in\Lambda$. 

Combining \eqref{notS} and \eqref{inS} for $v\in S$, we find infinitely many $n$ such that
\begin{align}
h(H_n)+nh(\beta_1,\hdots,\beta_t)=\sum_{v\in  M_{\kk}} \lambda_{H_{n},v}(\beta(n))<2\epsilon n.
\end{align}
This is impossible since $h(\beta_1,\hdots,\beta_t)>1$.

 It remains to establish the final conclusion of Proposition \ref{OS:2} (ii).  This is achieved via  the following observation.  If $m=\mathrm{o} (n) $, then for nonzero integers $a,b$, we have $|ma+nb|=|b|n+\mathrm{o} (n)$, contradicting the conclusion of (ii).
\end{proof}

\noindent{\bf Acknowledgments.} The authors are in debt to Pietro Corvaja for his insightful observation that led to an improvement of Theorem \ref{gcdrecurrencebasic}.
Both authors thank Aaron Levin for helpful comments and suggestions and for mentioning work-in-progress of Zheng Xiao which   studies Theorem \ref{gcdrecurrence:intro} and related results for linear recurrences.  They also thank Steven Lu for helping to facilitate this collaboration   and colleagues for their comments and interest in this work.
 Finally, we thank an anonymous referee for carefully reading this article and for providing  encouragement,  comments, suggestions and corrections.  Indeed, they helped us to improve upon our earlier results. 
This work began during the first author's visit to NCTS, Taipei, Taiwan.

 \end{document}